\documentclass[12pt]{article}
\usepackage{amssymb}
\usepackage{amsthm}
\usepackage{amsmath}
\usepackage[dvips]{epsfig}
\usepackage{graphicx}
\usepackage{color}
\usepackage{url}
\usepackage{tikz}
\usepackage{caption}
\usepackage{subcaption}
\usepackage{epstopdf}

\makeatletter\@addtoreset {equation}{section}\makeatother

\usetikzlibrary{matrix,arrows}
\newtheorem{thm}{Theorem}[section]

\newtheorem{prop}[thm]{Proposition}
\newtheorem{lem}[thm]{Lemma}

\newtheorem{remark}{Remark}

\newenvironment{proof1}%
{\begin{trivlist} \item[]{\em Proof }}%
{\hspace*{\fill}$\rule{.3\baselineskip}{.35\baselineskip}$\end{trivlist}}

\DeclareMathOperator{\sech}{sech}

\DeclareMathOperator{\Imag}{Im}
\DeclareMathOperator{\Real}{Re}

\begin{document}

\title{\bf Transverse instability of line solitons \\ in massive Dirac equations}

\author{Dmitry Pelinovsky and Yusuke Shimabukuro \\
{\small \it Department of Mathematics and Statistics, McMaster
University, Hamilton, Ontario, Canada, L8S 4K1 } }

\date{\today}
\maketitle

\begin{abstract}
Working in the context of localized modes in
periodic potentials, we consider two systems of the massive Dirac equations
in two spatial dimensions. The first system,  a generalized massive Thirring model,
is derived for the periodic stripe potentials. The second one, a generalized massive Gross--Neveu
equation, is derived for the hexagonal potentials. In both cases, we prove analytically
that the line solitons suffer from instability
with respect to periodic transverse perturbations of large periods.
The instability is induced by the spatial translation for the massive Thirring model
and by the gauge rotation for the massive Gross--Neveu model. We also observe numerically
that the instability holds for the transverse perturbations
of any period in the massive Thirring model and exhibits a finite threshold on
the period of the transverse perturbations in the massive Gross--Neveu model.
\end{abstract}

\section{Introduction}

Starting with pioneer contributions of V.E. Zakharov and his school \cite{ZR},
studies of transverse instabilities of line solitons in various nonlinear evolution equations
have been developed in many different contexts. With the exception of the Kadometsev--Petviashvili-II (KP-II)
equation, line solitons in many evolution equations suffer from instabilities
with respect to transverse periodic perturbations  \cite{KivPel}.

More recently, it was proved for the prototypical model of the KP-I equation  that the
line solitons under the transverse perturbations
of sufficiently small periods remain orbitally stable  \cite{RT12}. Similar thresholds on the period of transverse instability
exist in other models such as the elliptic nonlinear Schr\"{o}dinger (NLS) equation \cite{KK}
and the Zakharov--Kuznetsov (ZK) equation \cite{PRT}. Nevertheless, this conclusion is not universal
and the line solitons can be unstable for all periods of the transverse perturbations,
as it happens for the hyperbolic NLS equation \cite{PEOB}.

Conclusions on the transverse stability or instability of line solitons may change in the presence of
the periodic potentials. In the two-dimensional problems with square periodic potentials,
it was found numerically in \cite{Kev1,Kev2,Yang1} that transverse instability of line solitons is eliminated if the line
soliton bifurcates from the so-called $X$ point of the dispersion relation, where the homogeneous limit
is given by the hyperbolic NLS equation. Line solitons remain transversely unstable if they bifurcate
from the $\Gamma$ point of the dispersion relation,
where the homogeneous limit is given by the elliptic NLS equation.
These numerical results were rigorously justified in \cite{PY}
from the analysis of the two-dimensional discrete NLS equation, which
models the tight-binding limit of the periodic potentials \cite{PS}.

For the one-dimensional periodic (stripe) potentials, similar stabilization
of the transverse instability was observed numerically
in \cite{Yang2}. However, it was proven within the tight-binding limit in \cite{PY}
that transverse instabilities of line solitons persist for any parameter configurations
of the discrete NLS equation with continuous transverse dispersion. One of the motivations
for our present work is to study transverse stability of solitary waves in periodic stripe potentials
away from the tight--binding limit.

In particular, we employ the massive Dirac equations also known as the coupled-mode equations, which
have been derived and justified in the reduction of the Gross--Pitaevskii equation with small
periodic potentials \cite{US}. Similar models were also introduced in the context of the periodic stripe potentials
in \cite{AD}, where the primary focus was on the existence and stability
of fully localized two-dimensional solitary waves. From the class of massive Dirac models,
we will be particularly interested in the version of the massive Thirring model \cite{Thirring},
for which orbital stability of one-dimensional solitons was proved in our previous
work with the help of conserved quantities \cite{PS1} and auto--B\"{a}cklund transformation \cite{PS2}.
In the present work,  we prove analytically that the line solitons of the massive Thirring model
are prone to transverse instabilities with respect to the periodic perturbations of large periods
induced by the spatial translation. We also show numerically that the instability
persists for smaller periods of transverse perturbations.

Different versions of the massive Dirac equations were derived recently in the context of hexagonal potentials.
The corresponding systems generalize the massive Gross--Neveu model (also known as the Soler model in $(1+1)$ dimensions)
\cite{gross-neveu}. These equations were derived
formally in \cite{Abl1,Abl2} and were justified recently with rigorous analysis \cite{FW1,FW2,FW3}. Extending
the scope of our work, we prove analytically that the transverse instability of line solitons also holds for the
hexagonal potentials with respect to the periodic perturbations of large periods induced by the gauge rotation.
Numerical results indicate that the instability exhibits a finite threshold on
the period of the transverse perturbations.

The method we employ in our work is relatively old \cite{ZR} (see review in \cite{KivPel}),
although it has not been applied
to the class of massive Dirac equations even at the formal level.
We develop analysis at the rigorous level of arguments. Our work relies on the
resolvent estimates for the spectral stability problem in $(1+1)$ dimensions, where
the zero eigenvalue is disjoint from the continuous spectrum, whereas the eigenfunctions for
the zero eigenvalue are known from the translational and gauge symmetries of the massive Dirac equations.
When the transverse wave number is nonzero but small, the multiple zero eigenvalue split and
one can rigorously justify if this splitting induces the transverse instability or not.
It becomes notoriously more difficult to prove persistence of instabilities for large transverse
wave numbers (small periods), hence, we have to retreat to numerical computations for such studies of the
corresponding transverse stability problem.

The approach we undertake in this paper is complementary to the computations based on the Evans function approach \cite{Jon1,Jon2}.
Although both approaches stand on rigorous theory based on the implicit function theorem, we believe that
the perturbative computations are shorter and provide the binary answer on the transverse stability or instability of the
line soliton in a simple and concise way.

The structure of this paper is as follows. Section 2 introduces two systems of the massive Dirac equations
and their line solitons in the context of stripe and hexagonal potentials. Section 3 presents the analytical
results and gives details of algorithmic computations of the perturbation theory for the massive Thirring
and Gross--Neveu models. Section 4 contains numerical
approximations of eigenvalues of the transverse stability problem. Transverse instability of small-amplitude
line solitons in more general models is discussed in Section 5.

\section{Massive Dirac equations}

The class of massive Dirac equations in the space of one spatial dimension
can be written in the following general form \cite{CP,Pel-survey},
\begin{equation}
\label{dirac}
\left\{ \begin{array}{ll} i (u_t + u_x) + v  = \partial_{\bar{u}} W(u,v,\bar{u},\bar{v}), \\
i (v_t - v_x) + u = \partial_{\bar{v}} W(u,v,\bar{u},\bar{v}), \end{array} \right.
\end{equation}
where the subscripts denote partial differentiation, $(u,v)$ are complex-valued amplitudes
in spatial $x$ and temporal $t$ variables, and $W$ is the real function of $(u,v,\bar{u},\bar{v})$,
which is symmetric with respect to $u$ and $v$ and satisfies the gauge invariance
$$
W(e^{i \alpha} u,e^{i \alpha} v,e^{-i \alpha} \bar{u},e^{-i \alpha} \bar{v}) =
W(u,v,\bar{u},\bar{v}) \quad \mbox{\rm for every} \;\; \alpha \in \mathbb{R}.
$$
As it is shown in \cite{CP}, under the constraints on $W$, it can be expressed in terms of variables
$(|u|^2 + |v|^2)$, $|u|^2 |v|^2$, and $(\bar{u} v + u \bar{u})$.
For the cubic Dirac equations, $W$ is a homogeneous quartic polynomial in $u$ and $v$, which
can be expressed in the most general form as
$$
W = c_1 (|u|^2 + |v|^2)^2 + c_2 |u|^2 |v|^2 + c_3 (|u|^2 + |v|^2) (\bar{u} v + u \bar{u})
+ c_4 (\bar{u} v + u \bar{u})^2,
$$
where $c_1$, $c_2$, $c_3$, and $c_4$ are real coefficients.
In this case, a family of stationary solitary waves
of the massive Dirac equations can be found in the explicit form \cite{CP} (see also
\cite{Saxena1}).

Among various nonlinear Dirac equations, the following particular cases
have profound significance in relativity theory:
\begin{itemize}
\item $W = |u|^2 |v|^2$ - the massive Thirring model \cite{Thirring};
\item $W = \frac{1}{2} (\bar{u} v + u \bar{v})^2$ - the massive Gross--Neveu model \cite{gross-neveu}.
\end{itemize}
Global well-posedness of the massive Thirring model was proved
both in $H^s(\mathbb{R})$ for $s > \frac{1}{2}$ \cite{Selberg} and in $L^2(\mathbb{R})$ \cite{Candy}.
Recently, global well-posedness of the massive Gross--Neveu equations
was proved
both in $H^s(\mathbb{R})$ for $s > \frac{1}{2}$ \cite{Huh} and in $L^2(\mathbb{R})$ \cite{Zhang}.

When the massive Dirac equations are used in modeling of the
Gross--Pitaevskii equation with small periodic potentials,
the realistic nonlinear terms are typically different from the two
particular cases of the massive Thirring and Gross--Neveu models. (In this context, the nonlinear Dirac equations
are also known as the coupled-mode equations.) While computations
with the realistic nonlinear terms are important for applications,
they bring complexity that results in lack of clarity. In what follows,
we prefer to work with the two particular cases above. Nevertheless,
in the following two subsections,
we describe the connection of the massive Thirring and Gross--Neveu models  to physics of
nonlinear states of the Gross--Pitaevskii equation trapped in periodic potentials.

\subsection{Periodic stripe potentials}

In the context of one-dimensional periodic (stripe) potentials,
the massive Dirac equations (\ref{dirac}) can be derived in the following form \cite{AD},
\begin{equation}
\label{dirac1}
\left\{ \begin{array}{ll} i (u_t + u_x) + v + u_{yy} = (\alpha_1 |u|^2 + \alpha_2 |v|^2) u, \\
i (v_t - v_x) + u + v_{yy} = (\alpha_2 |u|^2 + \alpha_1 |v|^2) v, \end{array} \right.
\end{equation}
where $y$ is a new coordinate in the transverse direction to the stripe potential,
the complex-valued amplitudes $(u,v)$ correspond to two counter-propagating resonant Fourier modes
interacting with the small periodic potential,
and $(\alpha_1,\alpha_2)$ are real-valued parameters. For the stripe potentials,
the parameters satisfy the constraint $\alpha_2 = 2 \alpha_1$.
The massive Thirring model corresponds to the case $\alpha_1 = 0$ and $\alpha_2 \neq 0$ \cite{Thirring}.

To illustrate the derivation of the massive Dirac equations (\ref{dirac1}), we can
consider a two-dimensional Gross--Pitaevskii equation with a small periodic potential
\begin{equation}
\label{GPeq}
i \psi_t = -\psi_{xx} - \psi_{yy} + 2 \epsilon \cos(x) \psi + |\psi|^2 \psi,
\end{equation}
and apply the Fourier decomposition
\begin{equation}
\label{expansion-asympt}
\psi(x,y,t) = \sqrt{\epsilon} \left[ u(\epsilon x, \sqrt{\epsilon} y, \epsilon t) e^{\frac{i}{2} x - \frac{i}{4} t}
+ v(\epsilon x, \sqrt{\epsilon} y, \epsilon t) e^{-\frac{i}{2} x - \frac{i}{4} t} + \epsilon R(x,y,t) \right],
\end{equation}
where $\epsilon$ is a small parameter and $R$ is the remainder term. From the condition that $R$ is bounded
in variables $(x,y,t)$, it can be obtained from (\ref{GPeq}) and (\ref{expansion-asympt}) that
$(u,v)$ satisfy the nonlinear Dirac equations (\ref{dirac1})
with $\alpha_1 = 1$ and $\alpha_2 = 2$. Justification of the Fourier decomposition (\ref{expansion-asympt})
and the nonlinear Dirac equations (\ref{dirac1}) in the context of the Gross--Pitaevskii equation
(\ref{GPeq}) has been reported in \cite{US} (see also Chapter 2.2 in the book \cite{P}).

The stationary $y$-independent solitary waves of the massive Dirac equations (\ref{dirac1})
are referred to as the line solitons.
According to the analysis in \cite{CP,Saxena1}, the corresponding solutions
can be represented in the form
\begin{equation}
\label{line-soliton}
u(x,t) = U_{\omega}(x)e^{i \omega t}, \quad v(x,t) = \bar{U}_{\omega}(x) e^{i\omega t},
\end{equation}
where $\omega \in (-1,1)$ is taken in the gap between two branches of the
linear wave spectrum of the massive Dirac equations (\ref{dirac1}).
The complex-valued amplitude $U_{\omega}$ satisfies
the first-order differential equation
\begin{equation}
\label{equation-1}
i U_{\omega}' - \omega U_{\omega} + \overline{U}_{\omega} = (\alpha_1 + \alpha_2) |U_{\omega}|^2 U_{\omega}.
\end{equation}
In what follows, we simplify our presentation
and consider the particular configuration $\alpha_1 = 0$ and $\alpha_2 = 1$, which correspond to
the massive Thirring model. In this case, the solitary wave solution
exists for every $\omega \in (-1,1)$ in the explicit form
\begin{equation}
\label{MTM-soliton}
U_{\omega}(x) =\sqrt{2} \mu \frac{\sqrt{1+\omega}\cosh(\mu x)-i\sqrt{1-\omega}\sinh(\mu x)}{\omega + \cosh(2 \mu x)},
\end{equation}
where $\mu = \sqrt{1-\omega^2}$. As $\omega \to 1$, the family of solitary waves (\ref{MTM-soliton})
approaches the NLS profile $U_{\omega \to 1}(x) \to \mu {\rm sech}(\mu x)$.
As $\omega \to -1$, it degenerates into the algebraic profile
$$
U_{\omega = -1}(x) = \frac{2(1-2ix)}{1 + 4 x^2}.
$$

When $y$-independent perturbations are considered,
solitary waves (\ref{line-soliton}) and (\ref{MTM-soliton}) are orbitally stable
in the time evolution of the massive Thirring model for every $\omega \in (-1,1)$.
The corresponding results were obtained in our  previous works \cite{PS1} in
$H^1(\mathbb{R})$ and \cite{PS2} in a weighted subspace of $L^2(\mathbb{R})$.
Note that the solitary waves in more general nonlinear Dirac equations (\ref{dirac1})
are spectrally unstable for $y$-independent perturbations if $\alpha_1 \neq 0$
but the instability region and the number of unstable eigenvalues depend
on the parameter $\omega$ \cite{CP}.

\subsection{Hexagonal potentials}

In the context of the hexagonal potentials in two spatial dimensions,
the massive Dirac equations can be derived in a different form \cite{FW3},
\begin{eqnarray}\label{dirac2}
\left\{ \begin{array}{ll}
i\partial_t \varphi_1 + i\partial_{x}\varphi_2 - \partial_{y}\varphi_2 + \varphi_1 = (\beta_1|\varphi_1|^2+\beta_2|\varphi_2|^2) \varphi_1,\\
i\partial_t\varphi_2 + i\partial_{x}\varphi_1 + \partial_{y}\varphi_1 - \varphi_2 = (\beta_2|\varphi_1|^2+\beta_1|\varphi_2|^2)\varphi_2,
\end{array} \right.
\end{eqnarray}
where $(\varphi_1,\varphi_2)$ are complex-valued amplitudes for two resonant Floquet--Bloch
modes in the hexagonal lattice and $(\beta_1,\beta_2)$ are real-valued positive parameters.
The configuration of the massive Gross--Neveu model corresponds to the constraint
$\beta_1=-\beta_2$ \cite{gross-neveu}, which is only possible if the signs of
$\beta_1$ and $\beta_2$ are no longer both positive.

The nonlinear Dirac equations (\ref{dirac2}) correspond to equations (4.4)--(4.5) in \cite{FW3}.
Derivation of these equations can also be found in \cite{Abl1,Abl2}.
Justification of the linear part of these equations is performed by Fefferman and Weinstein \cite{FW2}.

To transform the nonlinear Dirac equations \eqref{dirac2}  to the form (\ref{dirac}),
we use the change of variables,
$$
\left(\begin{matrix}u\\v\end{matrix}\right)= \frac{1}{2}\left(\begin{matrix} 1& 1\\ 1& -1\end{matrix}\right) \left(\begin{matrix}\varphi_1\\\varphi_2\end{matrix}\right),
$$
and obtain
\begin{eqnarray}\label{dirac3}
\left\{ \begin{array}{ll} i (u_t + u_x) + v + v_y = \beta_1(u|u|^2+\overline{u}v^2+2u|v|^2) + \beta_2\overline{u}(u^2-v^2),\\
i (v_t - v_x) + u - u_y = \beta_1(v|v|^2+\overline{v}u^2+2v|u|^2) + \beta_2\overline{v}(v^2-u^2). \end{array} \right.
\end{eqnarray}
In comparison with the nonlinear Dirac equations (\ref{dirac1}),
we note that both the cubic nonlinearities and the $y$-derivative diffractive
terms are different.

For the family of line solitary waves (\ref{line-soliton}), the complex-valued amplitude $U_{\omega}$ satisfies
the first-order differential equation
\begin{equation}
\label{equation-2}
i U_{\omega}' - \omega U_{\omega} +\overline{U}_{\omega} = (3\beta_1+\beta_2)U_{\omega}|U_{\omega}|^2+(\beta_1-\beta_2)\overline{U}_{\omega}^3.
\end{equation}
In what follows, we simplify our presentation again
and consider the particular configuration $\beta_1 = - \beta_2 = \frac{1}{2}$, which corresponds to
the massive Gross--Neveu model. In this case, the solitary wave solution exists for every $\omega \in (0,1)$
in the explicit form
\begin{equation}
\label{Soler-soliton}
U_{\omega}(x) = \mu \frac{\sqrt{1+\omega}\cosh(\mu x)-i\sqrt{1-\omega}\sinh(\mu x)}{1 + \omega\cosh(2 \mu x)},
\end{equation}
where $\mu = \sqrt{1-\omega^2}$. The family of solitary waves (\ref{Soler-soliton})
diverges at infinity as $\omega \to 0$ and can not be continued for $\omega \in (-1,0)$ \cite{Comech}.
As $\omega \to 1$, the family approaches the
NLS profile $U_{\omega \to 1}(x) \to 2^{-1/2} \mu {\rm sech}(\mu x)$.

When $y$-independent perturbations are considered,
solitary waves (\ref{line-soliton}) and (\ref{Soler-soliton}) are orbitally stable in $H^1(\mathbb{R})$
in the time evolution of the massive Gross--Neveu model for $\omega \approx 1$ \cite{BC}.
Numerical approximations have showed contradictory results for smaller values of $\omega$.
A numerical approach based on the Evans function computation leads to the conclusion on
the spectral stability of solitary waves for all $\omega \in (0,1)$ \cite{Comech}.
However, another approach based on the finite-difference discretization indicates existence of $\omega_c \approx 0.6$
such that the family of solitary waves is spectrally stable for $\omega \in (\omega_c,1)$ and unstable for $\omega \in (0,\omega_c)$
\cite{Saxena1,Saxena2}. The presence of additional unstable eigenvalues in the case of $y$-independent perturbations,
if they exist, is not an obstacle in our analysis of transverse stability of line solitons.

\section{Transverse stability of line solitons}

We consider two versions (\ref{dirac1}) and (\ref{dirac3}) of
the nonlinear Dirac equations for spatial variables $(x,y)$ in the domain
$\mathbb{R}\times \mathbb{T}$, where $\mathbb{T}=\mathbb{R}/(L \mathbb{Z})$ is the one dimensional torus
and $L\in \mathbb{R}$ is the period of the transverse perturbation.
To study stability of the line solitary wave (\ref{line-soliton}) under periodic transverse perturbations,
we use the Fourier series and write
\begin{equation}
\label{Fourier-expansion}
u(x,y,t) = e^{i\omega t} \left[ U_{\omega}(x) + \sum_{n\in \mathbb{Z}} \hat{f}_{n}(x,t) e^{\frac{2\pi n i y}{L}} \right].
\end{equation}
In the settings of the linearized stability theory, we are going to use the linear superposition
principle and consider just one Fourier mode with continuous parameter $p \in \mathbb{R}$.
In the context of the Fourier series (\ref{Fourier-expansion}),
the parameter $p$ takes the countable set of
values $\{ \frac{2\pi n}{L} \}_{n \in \mathbb{Z}}$. Furthermore,
for each $p \in \mathbb{R}$, we separate the time evolution of the
linearized system and introduce the spectral parameter $\lambda$ in
the decomposition $\hat{f}_{n}(x,t) = \hat{F}_{n}(x) e^{\lambda t}$.
This decomposition reduces the linearized stability problem for $\hat{f}_n$
to the spectral stability problem for $\hat{F}_n$.

Performing similar manipulations with other components of the nonlinear Dirac equations,
we set the transverse perturbation in the form
\begin{eqnarray*}
u(x,y,t) = e^{i\omega t}[U_{\omega}(x)+u_1(x)e^{\lambda t+ipy}], & \quad &
\overline{u}(x,y,t) = e^{-i\omega t}[\overline{U}_{\omega}(x)+u_2(x)e^{\lambda t+ipy}],\\
v(x,y,t) = e^{i\omega t}[\overline{U}_{\omega}(x)+v_1(x)e^{\lambda t+ipy}], & \quad &
\overline{v}(x,y,t) = e^{-i\omega t}[U_{\omega}(x)+v_2(x)e^{\lambda t+ipy}],
\end{eqnarray*}
and obtain the spectral stability problem in the form
\begin{equation} \label{linear1}
i\lambda \sigma \mathbf{U}=(D_{\omega} + E_p + W_{\omega}) \mathbf{U},
\end{equation}
where $\mathbf{U}=(u_1,u_2,v_1,v_2)^t$, $\sigma=\text{diag}(1,-1,1,-1)$,
$$
D_{\omega} = \left[\begin{matrix}-i\partial_x+\omega &0&-1&0\\
0&i\partial_x+\omega&0&-1\\
-1&0&i\partial_x+\omega&0\\
0&-1&0&-i\partial_x+\omega
\end{matrix}\right],
$$
whereas matrices $E_p$ and $W_{\omega}$ depend on the particular form
of the nonlinear Dirac equations. For the massive Thirring model (\ref{dirac1})
with $\alpha_1 = 0$ and $\alpha_2 = 1$, we have $E_p = p^2 I$ with
\begin{equation}
\label{matrix-1}
I = \left[\begin{matrix} 1 &0&0&0\\
0&1&0&0\\
0&0&1&0\\
0&0&0&1
\end{matrix}\right] \quad \mbox{\rm and} \quad
W_{\omega}= \left[\begin{matrix}
|U_{\omega}|^2 & 0 & U_{\omega}^2 & |U_{\omega}|^2 \\
0 & |U_{\omega}|^2 & |U_{\omega}|^2 & \overline{U}_{\omega}^2 \\
\overline{U}_{\omega}^2 & |U_{\omega}|^2 & |U_{\omega}|^2 & 0 \\
|U_{\omega}|^2 & U_{\omega}^2 & 0 & |U_{\omega}|^2 \end{matrix}\right].
\end{equation}
For the massive Gross--Neveu model (\ref{dirac3}) with $\beta_1 = -\beta_2 = \frac{1}{2}$,
we obtain $E_p = -i p J$ with
\begin{equation}
\label{matrix-2}
J = \left[\begin{matrix}0 &0&1&0\\
0&0&0&1\\
-1&0&0&0\\
0&-1&0&0
\end{matrix}\right] \quad \mbox{\rm and} \quad
W_{\omega}= \left[\begin{matrix}
|U_{\omega}|^2 & \bar{U}_{\omega}^2 & U_{\omega}^2 + 2\overline{U}_{\omega}^2 & |U_{\omega}|^2 \\
U_{\omega}^2 & |U_{\omega}|^2 & |U_{\omega}|^2 & 2 U_{\omega}^2 + \overline{U}_{\omega}^2 \\
2 U_{\omega}^2 +  \overline{U}_{\omega}^2 &  |U_{\omega}|^2 &  |U_{\omega}|^2 & U_{\omega}^2 \\
 |U_{\omega}|^2 &  U_{\omega}^2 + 2 \overline{U}_{\omega}^2 & \overline{U}_{\omega}^2 & |U_{\omega}|^2 \end{matrix}\right].
\end{equation}

We note here that the linear operator $D_{\omega} + E_p + W_{\omega}$ is self-adjoint in $L^2(\mathbb{R},\mathbb{C}^4)$
with the domain in $H^1(\mathbb{R},\mathbb{C}^4)$ thanks to the boundness of the potential term $W_{\omega}$.
We shall use the notation $\langle \cdot,\cdot \rangle_{L^2}$ for the inner product in $L^2(\mathbb{R},\mathbb{C}^4)$
and the notation $\| \cdot \|_{L^2}$ for the induced norm. Also note that we apply complex conjugation
to the element at the first position of the inner product $\langle \cdot,\cdot \rangle_{L^2}$.

The next elementary result shows that the zero eigenvalue is isolated from the continuous spectrum
of the spectral stability problem (\ref{linear1}) both for $E_p = p^2 I$ and $E_p = -ip J$ if
the real parameter $p$ is sufficiently small.

\begin{prop}
\label{proposition-continuous-spectrum}
For every $p \in \mathbb{R}$, the continuous spectrum of the stability problem (\ref{linear1}) is located along
the segments $\pm i \Lambda_1$ and $\pm i \Lambda_2$, where for with $E_p = p^2 I$,
\begin{equation}
\label{continuous-1}
\Lambda_1 := \left\{ \sqrt{1 + k^2} + \omega + p^2, \quad k \in \mathbb{R} \right\}, \quad
\Lambda_2 := \left\{ \sqrt{1 + k^2} - \omega - p^2, \quad k \in \mathbb{R} \right\},
\end{equation}
whereas for $E_p = -ip J$,
\begin{equation}
\label{continuous-2}
\Lambda_1 := \left\{ \sqrt{1+p^2+k^2} + \omega, \quad k \in \mathbb{R} \right\}, \quad
\Lambda_2 := \left\{ \sqrt{1+p^2+k^2} - \omega, \quad k \in \mathbb{R} \right\}.
\end{equation}
\end{prop}

\begin{proof}
By Weyl's lemma, the continuous spectrum of the stability problem (\ref{linear1}) coincides
with the purely continuous spectrum of the same problem with $W_{\omega} \equiv 0$,
thanks to the exponential decay of the potential terms $W_{\omega}$ as $|x| \to \infty$.
If $W_{\omega} \equiv 0$, we solve the spectral stability problem (\ref{linear1})
with the Fourier transform in $x$, which means that
we simply replace $\partial_x$ in the operator $D_{\omega}$ with $ik$ for $k \in \mathbb{R}$
and denote the resulting matrix by $D_{\omega,k}$.
As a result, we obtain the matrix eigenvalue problem
$$
(D_{\omega,k} + E_p - i \lambda \sigma) \mathbf{U}=0.
$$
After elementary algebraic manipulations, the characteristic equation for
this linear system yields four solutions for $\lambda$ given by
$\pm i \Lambda_1$ and $\pm i \Lambda_2$, where the explicit expressions for $\Lambda_1$
and $\Lambda_2$ are given by (\ref{continuous-1}) and (\ref{continuous-2})
for $E_p = p^2 I$ and $E_p = -i p J$, respectively.
\end{proof}

\begin{remark}
We note the different role of the matrix $E_p$ in the location of the continuous spectrum
for larger values of the real parameter $p$. If $E_p = p^2 I$, then the two bands
$\pm i \Lambda_2$ touches each other for $|p| = p_{\omega} := \sqrt{1-\omega}$
and overlap for $|p| > p_{\omega}$. If $E_p = -ipJ$, all the four bands do not
overlap for all values of $p \in \mathbb{R}$ and the zero point is always in the gap
between the branches of the continuous spectrum.
\end{remark}

The next result shows that if $p = 0$, then the spectral stability
problem (\ref{linear1}) admits the zero eigenvalue of quadruple multiplicity. The zero eigenvalue is
determined by the symmetries of the nonlinear Dirac equations
with respect to the spatial translation and the gauge rotation.

\begin{prop}
\label{proposition-kernel}
If $p = 0$, the stability problem (\ref{linear1}) admits exactly two eigenvectors in $H^1(\mathbb{R})$
for the eigenvalue $\lambda = 0$ given by
\begin{equation}
\label{zero-1}
{\bf U}_t = \partial_x {\bf U}_{\omega}, \quad {\bf U}_g = i \sigma {\bf U}_{\omega},
\end{equation}
where ${\bf U}_{\omega} = (U_{\omega},\bar{U}_{\omega},\bar{U}_{\omega},U_{\omega})^t$.
For each eigenvector ${\bf U}_{t,g}$, there exists a generalized eigenvector
$\tilde{\bf U}_{t,g}$ in $H^1(\mathbb{R})$ from solutions of the
inhomogeneous problem
\begin{equation}
\label{inhomogeneous-zero-1}
(D_{\omega} + W_{\omega}) {\bf U} = i \sigma {\bf U}_{t,g},
\end{equation}
in fact, in the explicit form,
\begin{equation}
\label{zero-2}
\tilde{\bf U}_{t} = i \omega x \sigma {\bf U}_{\omega} - \frac{1}{2} \tilde{\sigma} {\bf U}_{\omega}, \quad
\tilde{\bf U}_{g} = \partial_{\omega} {\bf U}_{\omega},
\end{equation}
where $\tilde{\sigma} = \mbox{\rm diag}(1,1,-1,-1)$.
Moreover, for the generalized eigenvector
$\tilde{\bf U}_{t,g}$, no solutions of
the inhomogeneous problem
\begin{equation}
\label{inhomogeneous-zero-2}
(D_{\omega} + W_{\omega}) {\bf U} = i \sigma \tilde{\bf U}_{t,g}
\end{equation}
exist in $H^1(\mathbb{R})$.
\end{prop}

\begin{proof}
Existence of the eigenvectors (\ref{zero-1}) follows from the two symmetries
of the massive Dirac equations and is checked by elementary substitution as
$(D_{\omega} + W_{\omega}) {\bf U}_{t,g} = {\bf 0}$.

Because $(D_{\omega} + W_{\omega})$
is a self-adjoint operator of the fourth order and solutions of $(D_{\omega} + W_{\omega}) {\bf U} = {\bf 0}$
have constant Wronskian determinant in $x$, there exists at most two spatially decaying
solutions of these homogeneous equations, which means that the stability problem (\ref{linear1}) with $p = 0$
admits exactly two eigenvectors in $H^1(\mathbb{R})$. Since
$$
\langle {\bf U}_{t,g}, \sigma {\bf U}_{t,g} \rangle_{L^2} =
\langle {\bf U}_{t,g}, \sigma {\bf U}_{g,t} \rangle_{L^2} = 0
$$
there exist solutions of the inhomogeneous problem (\ref{inhomogeneous-zero-1}) in $H^1(\mathbb{R})$.
Existence of the generalized eigenvectors (\ref{zero-2}) is checked by elementary substitution.
Finally, we have
$$
\langle {\bf U}_{t,g}, \sigma \tilde{\bf U}_{t,g} \rangle_{L^2} \neq 0, \quad
\langle {\bf U}_{t,g}, \sigma \tilde{\bf U}_{g,t} \rangle_{L^2} = 0,
$$
therefore, no solutions of the inhomogeneous problem (\ref{inhomogeneous-zero-2}) exist in $H^1(\mathbb{R})$.
\end{proof}

Our main result is formulated in the following theorem. The theorem guarantees instability of the
line solitary waves with respect to the transverse perturbations of sufficiently large period
both for the massive Thirring model and the massive Gross--Neveu model.

\begin{thm}\label{specThm}
Let $\mathcal{O} \subset (-1,1)$ be the existence interval for
the line solitary wave (\ref{line-soliton}) of the nonlinear Dirac equations (\ref{dirac}).
For every $\omega \in \mathcal{O}$, there exists $p_0 > 0$ such that
for every $p$ in $0 < |p| < p_0$, the spectral stability problem \eqref{linear1}
with either (\ref{matrix-1}) or (\ref{matrix-2}) admits at least one eigenvalue $\lambda$ with ${\rm Re}(\lambda) > 0$.
Moreover, up to a suitable normalization, as $p \to 0$,
the corresponding eigenvector ${\bf U}$ converges in $L^2(\mathbb{R})$ to ${\bf U}_t$ for the massive Thirring model with (\ref{matrix-1})
and to ${\bf U}_g$ for the massive Gross--Neveu model with (\ref{matrix-2}).

Simultaneously, there exists at least one pair of purely imaginary
eigenvalues $\lambda$ of the spectral stability problem \eqref{linear1}
and the corresponding eigenvector ${\bf U}$ converges as $p \to 0$ to
the other eigenvector of Proposition \ref{proposition-kernel}.
\end{thm}

The proof of Theorem \ref{specThm} is based on the perturbation theory for the
Jordan block associated with the zero eigenvalue of the spectral
problem (\ref{linear1}) existing for $p = 0$, according to
Proposition \ref{proposition-kernel}. The zero eigenvalue is isolated from the continuous spectrum,
according to Proposition \ref{proposition-continuous-spectrum}.
Consequently, we do not have to deal with bifurcations from the continuous spectrum
(unlike the difficult tasks of the recent work \cite{BC}), but
can develop straightforward perturbation expansions based on a modification
of the Lyapunov--Schmidt reduction method.

A useful technical approach to the perturbation theory
for the spectral stability problem (\ref{linear1}) is based on
the block diagonalization of the $4\times 4$ matrix operator into two
$2 \times 2$ Dirac operators. This block diagonalization technique was introduced
in \cite{CP} and used for numerical approximations of eigenvalues of the spectral stability
problem for the massive Dirac equations. After the block diagonalization, each Dirac operator has a one-dimensional kernel space
induced by either translational or gauge symmetries. It enables us to uncouple
the invariant subspaces associated with the Jordan block for the zero eigenvalue
of the spectral stability problem (\ref{linear1}) with $p = 0$.

Using the self-similarity transformation matrix
$$
S=\frac{1}{\sqrt{2}}\left[\begin{matrix}1&0&1&0\\ 0&1&0&1\\0&1&0&-1\\1&0&-1&0  \end{matrix}\right]
$$
and setting $\mathbf{U} = S \mathbf{V}$, we can rewrite the
spectral stability problem \eqref{linear1} in the following form:
\begin{equation} \label{linear2}
i \lambda S^t \sigma S \mathbf{V}= S^t (D_{\omega} + E_p + W_{\omega}) S \mathbf{V},
\end{equation}
where
\begin{equation}
\label{matrix-3-4}
S^t D_{\omega} S = \left[\begin{matrix}-i\partial_x+\omega & -1 & 0 & 0\\
-1 & i\partial_x+\omega&0&0\\ 0&0& -i\partial_x+\omega&1\\
0&0&1&i\partial_x+\omega
\end{matrix}\right], \quad
S^t \sigma S = \left[\begin{matrix} 0 & 0 & 1 & 0\\ 0 & 0 & 0 & -1\\ 1 & 0 & 0 & 0 \\ 0 & -1 & 0 & 0 \end{matrix}\right],
\end{equation}
whereas the transformation of matrices $E_p$ and $W_{\omega}$ depend on the particular form
of the nonlinear Dirac equations. For the massive Thirring model (\ref{dirac1})
with $\alpha_1 = 0$ and $\alpha_2 = 1$, we have
\begin{equation}
\label{matrix-3}
S^t E_p S = p^2 \left[\begin{matrix} 1 &0&0&0\\ 0&1&0&0\\ 0&0&1&0\\ 0&0&0&1 \end{matrix}\right], \quad
S^t W_{\omega} S = \left[\begin{matrix} 2 |U_{\omega}|^2 & U_{\omega}^2 & 0 & 0  \\
\overline{U}_{\omega}^2 & 2 |U_{\omega}|^2 & 0 & 0 \\ 0 & 0 & 0 & - U_{\omega}^2 \\
0 & 0 & -\overline{U}_{\omega}^2 & 0 \end{matrix}\right].
\end{equation}
For the massive Gross--Neveu model (\ref{dirac3}) with $\beta_1 = -\beta_2 = \frac{1}{2}$,
we obtain
\begin{equation}
\label{matrix-4}
S^t E_p S = i p \left[\begin{matrix}0 &0&0&1\\ 0&0&1&0\\
0&-1&0&0\\ -1&0&0&0 \end{matrix}\right], \quad
S^t W_{\omega} S = \left[\begin{matrix}
2 |U_{\omega}|^2 & U_{\omega}^2 + 3 \bar{U}_{\omega}^2 & 0 & 0 \\
3 U_{\omega}^2 + \bar{U}_{\omega}^2 & 2 |U_{\omega}|^2 & 0 & 0 \\
0 & 0 & 0 & -U_{\omega}^2-\bar{U}_{\omega}^2 \\
0 & 0 & -U_{\omega}^2-\bar{U}_{\omega}^2 & 0 \end{matrix}\right].
\end{equation}

Let us apply the self-similarity transformation to the eigenvectors and generalized eigenvectors of
Proposition \ref{proposition-kernel}. Using $\mathbf{U} = S \mathbf{V}$, the eigenvectors
(\ref{zero-1}) become
\begin{equation}
\label{kernel-1}
\mathbf{V}_t=\left(\begin{matrix} U_{\omega}'\\\overline{U}_{\omega}' \\ 0 \\ 0 \end{matrix}\right) \quad \mbox{and} \quad
\mathbf{V}_g=i\left(\begin{matrix} 0 \\ 0 \\U_{\omega}\\-\overline{U}_{\omega}\end{matrix}\right),
\end{equation}
whereas the generalized eigenvectors (\ref{zero-2}) become
\begin{equation}
\label{kernel-2}
\tilde{\mathbf{V}}_t = i \omega x \left(\begin{matrix} 0 \\ 0 \\ U_{\omega}\\-\overline{U}_{\omega}\end{matrix}\right) -
\frac{1}{2} \left(\begin{matrix} 0 \\ 0 \\ U_{\omega}\\ \overline{U}_{\omega}\end{matrix}\right) \quad \mbox{and} \quad
\tilde{\mathbf{V}}_g = \partial_{\omega} \left(\begin{matrix}U_{\omega}\\ \overline{U}_{\omega} \\ 0 \\ 0\end{matrix}\right).
\end{equation}

Setting $\Phi_V = [{\bf V}_{t}, {\bf V}_{g}, \tilde{\bf V}_{t}, \tilde{\bf V}_{g}]$
and denoting $\mathcal{S} = S^t \sigma S$, we compute elements of the matrix of skew-symmetric inner products
between eigenvectors and generalized eigenvectors:
\begin{equation}
\label{projections-matrix-2}
\langle \Phi_V, \mathcal{S} \Phi_V \rangle_{L^2} = \left[ \begin{array}{cccc} 0 & 0 &
\langle {\bf V}_{t}, \mathcal{S} \tilde{\bf V}_{t} \rangle_{L^2} & 0 \\
0 & 0 & 0 & \langle {\bf V}_{g}, \mathcal{S} \tilde{\bf V}_{g} \rangle_{L^2} \\
\langle \tilde{\bf V}_{t}, \mathcal{S} {\bf V}_{t} \rangle_{L^2} & 0 & 0 & 0 \\
0 & \langle \tilde{\bf V}_{g}, \mathcal{S} {\bf V}_{g} \rangle_{L^2} & 0 & 0 \end{array} \right],
\end{equation}
where only nonzero elements are included. Verification of (\ref{projections-matrix-2}) is straightforward except for
the term
\begin{equation}
\label{projections-matrix-0}
\langle \tilde{\bf V}_{t}, \mathcal{S} \tilde{\bf V}_{g} \rangle_{L^2} = - i \omega \int_{\mathbb{R}} x \partial_{\omega} |U_{\omega}|^2 dx
- \frac{1}{2} \int_{\mathbb{R}} \left( \bar{U}_{\omega} \partial_{\omega} U_{\omega} - U_{\omega} \partial_{\omega} \bar{U}_{\omega} \right) dx = 0.
\end{equation}
Both integrals in (\ref{projections-matrix-0}) are zero because
$x \partial_{\omega} |U_{\omega}|^2$ and ${\rm Im}(\bar{U}_{\omega} \partial_{\omega} U_{\omega})$ are odd functions of $x$.
As for the nonzero elements, we compute them explicitly from (\ref{kernel-1}) and (\ref{kernel-2}):
\begin{equation}
\label{nonzero-element-1}
\langle {\bf V}_{t}, \mathcal{S} \tilde{\bf V}_{t} \rangle_{L^2} = - i \omega \int_{\mathbb{R}} |U_{\omega}|^2 dx + \frac{1}{2}
\int_{\mathbb{R}} \left( \bar{U}_{\omega} U_{\omega}' - U_{\omega} \bar{U}_{\omega}' \right) dx
\end{equation}
and
\begin{equation}
\label{nonzero-element-2}
\langle {\bf V}_{g}, \mathcal{S} \tilde{\bf V}_{g} \rangle_{L^2} = - i \frac{d}{d\omega} \int_{\mathbb{R}} |U_{\omega}|^2 dx.
\end{equation}

We shall now proceed separately with the proof of Theorem \ref{specThm} for the massive Thirring and Gross--Neveu models.
Moreover, we derive explicit asymptotic expressions for the eigenvalues mentioned in Theorem \ref{specThm}.

\subsection{Perturbation theory for the massive Thirring model}

In the case of the massive Thirring model with (\ref{matrix-3-4}) and (\ref{matrix-3}),
the block-diagonalized system (\ref{linear2}) can be rewritten in the explicit form
\begin{equation} \label{eqMTM}
\left( \begin{matrix} H_+ & 0 \\ 0 & H_- \end{matrix} \right) {\bf V}
+ p^2 \left( \begin{matrix} \sigma_0 & 0 \\ 0 & \sigma_0 \end{matrix} \right) {\bf V}
= i \lambda \left( \begin{matrix} 0 & \sigma_3 \\ \sigma_3 & 0 \end{matrix} \right) {\bf V},
\end{equation}
where
\begin{equation}
H_+ = \left(\begin{matrix}-i\partial_x+\omega+2|U_{\omega}|^2&-1+U_{\omega}^2 \\
-1 + \overline{U}_{\omega}^2 & i\partial_x+\omega+2|U_{\omega}|^2 \end{matrix}\right), \quad
H_- = \left(\begin{matrix}-i\partial_x+\omega & 1-U_{\omega}^2 \\
1-\overline{U}_{\omega}^2  &i\partial_x+\omega \end{matrix}\right),
\end{equation}
and the following Pauli matrices are used throughout our work:
\begin{equation}
\sigma_0 = \left( \begin{array}{cc} 1 & 0 \\ 0 & 1 \end{array} \right), \quad
\sigma_1 = \left( \begin{array}{cc} 0 & 1 \\ 1 & 0 \end{array} \right), \quad
\sigma_3 = \left( \begin{array}{cc} 1 & 0 \\ 0 & -1 \end{array} \right).
\end{equation}
Note that $H_+$ and $H_-$ are self-adjoint operators in $L^2(\mathbb{R},\mathbb{C}^2)$
with the domain in $H^1(\mathbb{R},\mathbb{C}^2)$. The operators $H_{\pm}$ satisfy
the symmetry
\begin{equation}
\label{symmetry-Dirac-1}
\sigma_1 H_{\pm} = \bar{H}_{\pm} \sigma_1,
\end{equation}
whereas the Pauli matrices satisfy the relation
\begin{equation}
\label{symmetry-Dirac-2}
\sigma_1 \sigma_1 = \sigma_3 \sigma_3 = \sigma_0, \quad \sigma_1 \sigma_3 + \sigma_3 \sigma_1 = 0,
\end{equation}
Before proving the main result of the perturbation theory for the massive Thirring model,
we note the following elementary result.

\begin{prop}
\label{proposition-symmetry}
For every $p \in \mathbb{R}$, eigenvalues $\lambda$ of the spectral problem (\ref{eqMTM})
are symmetric about the real and imaginary axes in the complex plane.
\end{prop}

\begin{proof}
It follows from symmetries (\ref{symmetry-Dirac-1}) and (\ref{symmetry-Dirac-2}) that
if $\lambda$ is an eigenvalue of the spectral problem (\ref{eqMTM}).
with the eigenvector ${\bf V} = (v_1,v_2,v_3,v_4)^t$, then $\bar{\lambda}$, $-\lambda$, and $-\bar{\lambda}$ are
also eigenvalues of the same problem with the eigenvectors $(\bar{v}_2,\bar{v}_1,\bar{v}_4,\bar{v}_3)^t$,
$(v_1,v_2,-v_3,-v_4)^t$, and $(\bar{v}_2,\bar{v}_1,-\bar{v}_4,-\bar{v}_3)^t$.
Consequently, we have the following:
\begin{itemize}
\item if $\lambda$ is a simple real nonzero eigenvalue, then the eigenvector ${\bf V}$ can be chosen to satisfy the reduction
$v_1 = \bar{v}_2$, $v_3 = \bar{v}_4$, whereas $-\lambda$ is also an eigenvalue with the eigenvector
$(v_1,v_2,-v_3,-v_4)^t = (\bar{v}_2,\bar{v}_1,-\bar{v}_4,-\bar{v}_3)^t$;
\item if $\lambda$ is a simple purely imaginary nonzero eigenvalue, then the eigenvector ${\bf V}$ can be chosen to satisfy the reduction
$v_1 = \bar{v}_2$, $v_3 = -\bar{v}_4$, whereas $\bar{\lambda}$ is also an eigenvalue with the eigenvector
$(\bar{v}_2,\bar{v}_1,\bar{v}_4,\bar{v}_3)^t = (v_1,v_2,-v_3,-v_4)^t$;
\item if a simple eigenvalue $\lambda$ occurs in the first quadrant, then the symmetry generates eigenvalues in
all other quadrants and all four eigenvectors generated by the symmetry are linearly independent.
\end{itemize}
The symmetry between eigenvalues also applies to multiple nonzero eigenvalues and the corresponding eigenvectors
of the associated Jordan blocks.
\end{proof}

For the sake of simplicity, we denote
$$
\mathcal{H} = \left( \begin{matrix} H_+ & 0 \\ 0 & H_- \end{matrix} \right), \quad
\mathcal{I} = \left( \begin{matrix} \sigma_0 & 0 \\ 0 & \sigma_0 \end{matrix} \right), \quad
\mathcal{S} = \left( \begin{matrix} 0 & \sigma_3 \\ \sigma_3 & 0 \end{matrix} \right).
$$
Setting $\Phi_V = [{\bf V}_{t}, {\bf V}_{g}, \tilde{\bf V}_{t}, \tilde{\bf V}_{g}]$ as earlier,
we note that
\begin{equation}
\label{projections-matrix-3}
\langle \Phi_V, \mathcal{I} \Phi_V \rangle_{L^2} = \left[ \begin{array}{cccc} \| {\bf V}_t \|_{L^2}^2 & 0 & 0 & 0 \\
0 & \| {\bf V}_g \|_{L^2}^2 &  0 & 0 \\ 0 & 0  & \| \tilde{\bf V}_t \|_{L^2}^2 & 0 \\
0 & 0 & 0 & \| \tilde{\bf V}_g \|_{L^2}^2 \end{array} \right],
\end{equation}
where only nonzero terms are included. Again, verification of (\ref{projections-matrix-3})
follows straightforwardly from (\ref{kernel-1}) and (\ref{kernel-2}) except for the elements
$$
\langle {\bf V}_{t}, \tilde{\bf V}_{g} \rangle_{L^2} = \int_{\mathbb{R}} \left(
\bar{U}_{\omega}' \partial_{\omega} U_{\omega} + U_{\omega}' \partial_{\omega} \bar{U}_{\omega} \right) dx = 0
$$
and
$$
\langle {\bf V}_{g}, \tilde{\bf V}_{t} \rangle_{L^2} = 2 \omega \int_{\mathbb{R}} x |U_{\omega}|^2 dx = 0.
$$
These elements are zero because $x |U_{\omega}|^2$ and ${\rm Re}(\bar{U}_{\omega}' \partial_{\omega} U_{\omega})$ are odd functions of $x$.

\begin{figure}[b!]
        \centering
               \begin{subfigure}[htbp]{0.4\textwidth}
  \includegraphics[scale=0.5]{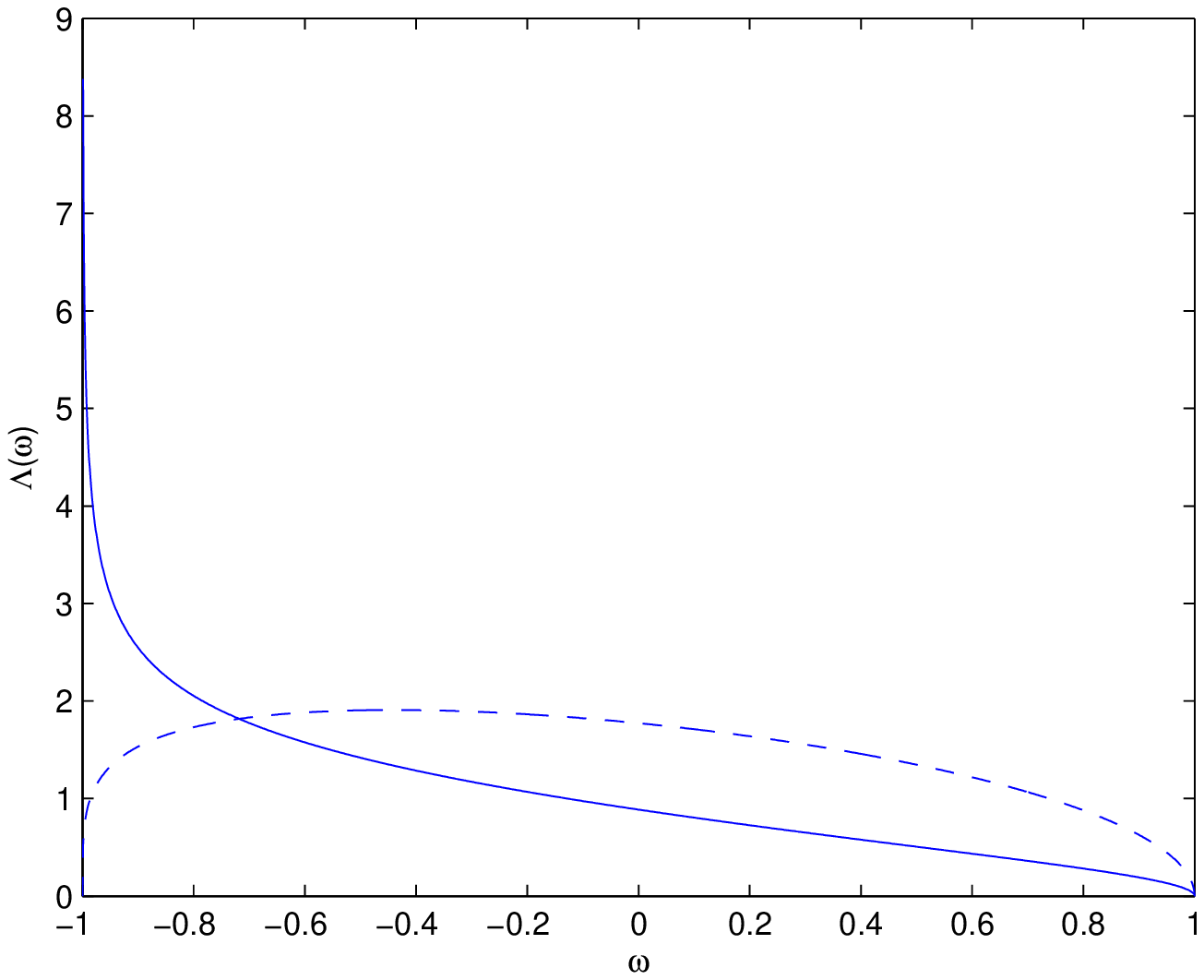}
             \caption{}
                \label{FigTheoryA}
  \end{subfigure}
    \quad \quad \quad \quad \quad
         \begin{subfigure}[htbp]{0.4\textwidth}
   \includegraphics[scale=0.5]{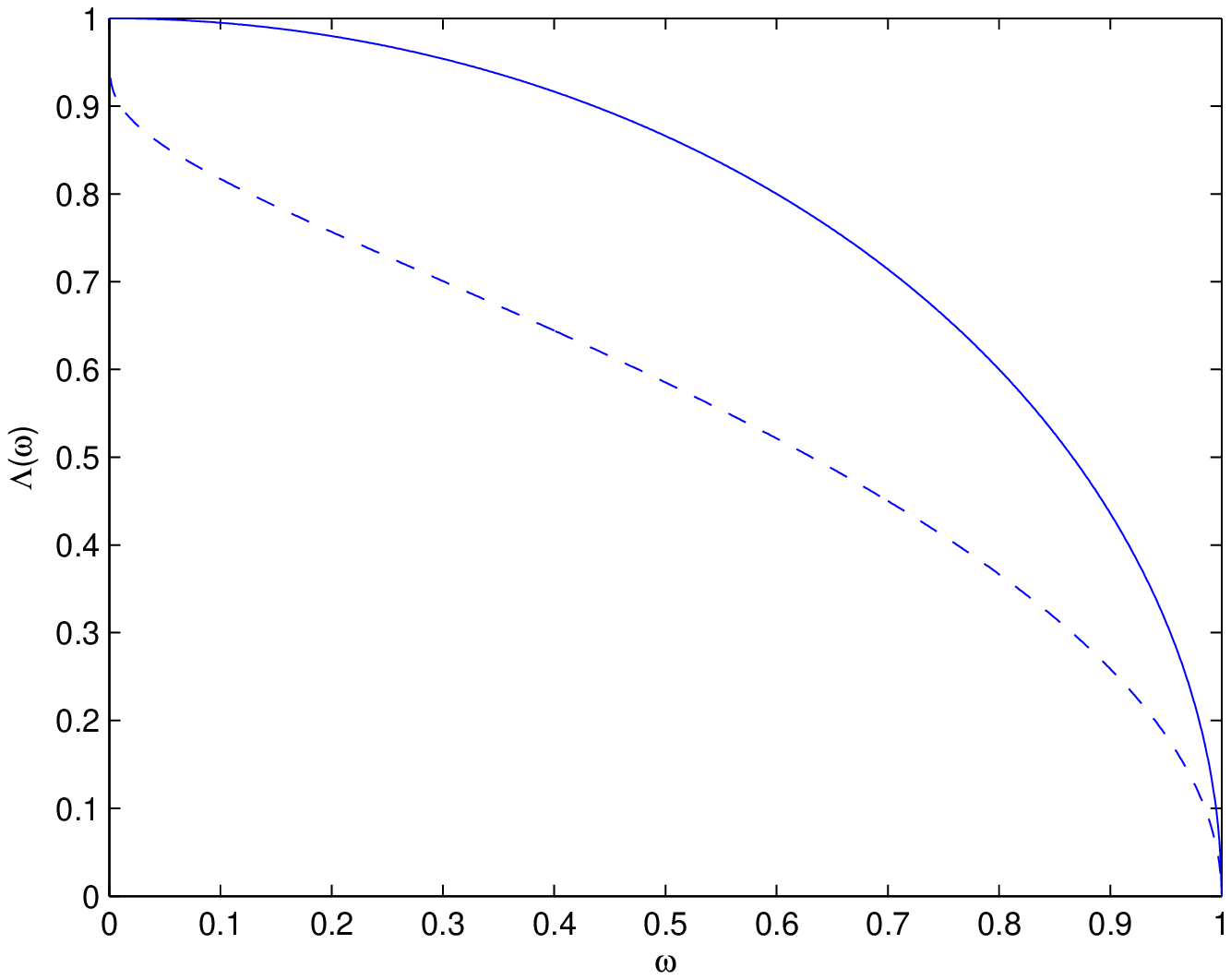}
              \caption{}
                \label{FigTheoryB}
   \end{subfigure}
             \caption{Asymptotic expressions $\Lambda_r$ (solid line) and $\Lambda_i$ (dashed line)
             versus parameter $\omega$ for the massive Thirring (left) and Gross--Neveu (right) models.}
             \label{FigTheory}
\end{figure}

The following result gives the outcome of the perturbation theory
associated with the generalized null space of the spectral stability
problem (\ref{eqMTM}). The result is equivalent to the part of Theorem \ref{specThm}
corresponding to the massive Thirring model. The asymptotic expressions
$\Lambda_r$ and $\Lambda_i$ of the real and imaginary
eigenvalues $\lambda$ at the leading order in $p$
versus parameter $\omega$ are shown on Fig. \ref{FigTheoryA}.

\begin{lem}\label{lemmaMTM}
For every $\omega \in (-1,1)$, there exists $p_0 > 0$ such that
for every $p$ with $0 < |p| < p_0$, the spectral stability problem \eqref{eqMTM}
admits a pair of real eigenvalues $\lambda$ with the eigenvectors ${\bf V} \in H^1(\mathbb{R})$
such that
\begin{equation}
\label{asymptMTM-1}
\lambda = \pm p \Lambda_r(\omega) + \mathcal{O}(p^3), \quad {\bf V} = \mathbf{V}_t
\pm p \Lambda_r(\omega) \tilde{\mathbf{V}}_t + \mathcal{O}_{H^1}(p^2) \quad \mbox{\rm as} \quad p \to 0,
\end{equation}
where $\Lambda_r = (1-\omega^2)^{-1/4} \| U_{\omega}' \|_{L^2} > 0$.
Simultaneously, it admits a pair of purely imaginary eigenvalues $\lambda$ with
the eigenvector ${\bf V} \in H^1(\mathbb{R})$
such that
\begin{equation}
\label{asymptMTM-2}
\lambda = \pm i p \Lambda_i(\omega) + \mathcal{O}(p^3), \quad
{\bf V} = \mathbf{V}_g \pm i p \Lambda_i(\omega) \tilde{\mathbf{V}}_g + \mathcal{O}_{H^1}(p^2) \quad \mbox{\rm as} \quad p \to 0,
\end{equation}
where  $\Lambda_i = \sqrt{2} (1-\omega^2)^{1/4} \| U_{\omega} \|_{L^2} > 0$.
\end{lem}

Before proving Lemma \ref{lemmaMTM}, we give formal computations of the perturbation theory,
which recover expansions (\ref{asymptMTM-1}) and (\ref{asymptMTM-2})
with explicit expressions for $\Lambda_r(\omega)$ and $\Lambda_i(\omega)$.
Consider the following formal expansions
\begin{equation}
\label{asymptMTM-3}
\lambda = p \Lambda_1 + p^2 \Lambda_2 + \mathcal{O}(p^3), \quad
{\bf V} = {\bf V}_0 + p \Lambda_1 {\bf V}_1 + p^2 {\bf V}_2 + \mathcal{O}_{H^1}(p^3),
\end{equation}
where ${\bf V}_0$ is spanned by the eigenvectors (\ref{kernel-1}), ${\bf V}_1$ is
spanned by the generalized eigenvectors (\ref{kernel-2}), and ${\bf V}_2$ satisfies
the linear inhomogeneous equation
\begin{equation}
\label{linear3}
\mathcal{H} {\bf V}_2 = - \mathcal{I} {\bf V}_0 + i \Lambda_1^2 \mathcal{S} {\bf V}_1
+ i \Lambda_2 \mathcal{S} {\bf V}_0.
\end{equation}
By the Fredholm alternative, there exists a solution ${\bf V}_2 \in H^1(\mathbb{R})$ of
the linear inhomogeneous equation (\ref{linear3}) if and only if
$\Lambda_1$ is found from the quadratic equation
\begin{equation}
\label{solvability3}
i \Lambda_1^2 \langle {\bf W}_0, \mathcal{S} {\bf V}_1
\rangle_{L^2} = \langle {\bf W}_0, {\bf V}_0 \rangle_{L^2},
\end{equation}
where ${\bf W}_0$ is spanned by the eigenvectors of $\mathcal{H}$ independently of ${\bf V}_0$.
Because of the block diagonalization of the projection matrices in (\ref{projections-matrix-2}) and (\ref{projections-matrix-3}),
the $2$-by-$2$ matrix eigenvalue problem (\ref{solvability3}) is diagonal and
we can proceed separately for each eigenvector in ${\bf V}_0$.

Selecting ${\bf V}_0 = {\bf W}_0 = \mathbf{V}_t$ and
${\bf V}_1 = \tilde{\mathbf{V}}_t$, we rewrite the solvability condition (\ref{solvability3})
as the following quadratic equation
$$
\Lambda_1^2 \int_{\mathbb{R}} \left( \omega |U_{\omega}|^2 + \frac{i}{2}
\left( \bar{U}_{\omega} U_{\omega}' - U_{\omega} \bar{U}_{\omega}' \right) \right) dx =
2 \int_{\mathbb{R}} |U_{\omega}'|^2 dx
$$
where we have used relation (\ref{nonzero-element-1}).
Substituting the exact expression (\ref{MTM-soliton}), we obtain
$$
\int_{\mathbb{R}} \left( \omega |U_{\omega}|^2 + \frac{i}{2}
\left( \bar{U}_{\omega} U_{\omega}' - U_{\omega} \bar{U}_{\omega}' \right) \right) dx =
2 \sqrt{1-\omega^2}
$$
and
$$
\int_{\mathbb{R}} |U_{\omega}'|^2 dx = -4\omega \sqrt{1-\omega^2}+4(1+\omega^2) \arctan\left(\sqrt{\frac{1-\omega}{1+\omega}}\right),
$$
which yields the expression $\Lambda_1^2 = (1-\omega^2)^{-1/2} \| U_{\omega}' \|_{L^2}^2 = \Lambda_r(\omega)^2$.

Selecting now ${\bf V}_0 = {\bf W}_0 = \mathbf{V}_g$ and
${\bf V}_1 = \tilde{\mathbf{V}}_g$,
we rewrite the solvability condition (\ref{solvability3})
as the following quadratic equation
$$
\Lambda_1^2 \frac{d}{d\omega} \int_{\mathbb{R}} |U_{\omega}|^2 dx =
2 \int_{\mathbb{R}} |U_{\omega}|^2 dx,
$$
where we have used relation (\ref{nonzero-element-2}).
Substituting the exact expression (\ref{MTM-soliton}), we obtain
$$
\int_{\mathbb{R}} |U_{\omega}|^2 dx = 4\arctan\left(\sqrt{\frac{1-\omega}{1+\omega}}\right)
$$
and
$$
\frac{d}{d\omega} \int_{\mathbb{R}} |U_{\omega}|^2 dx = -\frac{1}{\sqrt{1-\omega^2}},
$$
which yields the expression for $\Lambda_1^2 = -2 (1-\omega^2)^{1/2} \| U_{\omega} \|_{L^2}^2 = -\Lambda_i(\omega)^2$.

We shall now justify the formal expansions (\ref{asymptMTM-1}) and (\ref{asymptMTM-2}) to give the proof of Lemma \ref{lemmaMTM}.
Note that $\Lambda_2$ in (\ref{asymptMTM-3}) is not determined in the linear inhomogeneous equation (\ref{linear3}).
Nevertheless, we will show in the proof of Lemma \ref{lemmaMTM} that $\Lambda_2 = 0$.

\begin{proof1}{\em of Lemma \ref{lemmaMTM}.}
Consider the linearized operator for the spectral problem (\ref{eqMTM}):
$$
\mathcal{A}_{\lambda,p} = \mathcal{H} + p^2 \mathcal{I} - i \lambda \mathcal{S} : H^1(\mathbb{R}) \to L^2(\mathbb{R}).
$$
This operator is self-adjoint if $\lambda \in i \mathbb{R}$ and nonself-adjoint if $\lambda \notin i \mathbb{R}$.
If $p = 0$, then $\mathcal{A}_{\lambda,0}$ has the four-dimensional generalized null space $X_0 \subset L^2(\mathbb{R})$
spanned by the vectors in $\Phi_V$. By Propositions \ref{proposition-continuous-spectrum} and \ref{proposition-kernel},
the rest of spectrum of $\mathcal{A}_{\lambda,0}$ is bounded away from
zero. Consequently, there is $\lambda_0 > 0$ sufficiently small such that
$\mathcal{A}_{\lambda,0}$ with $|\lambda| < \lambda_0$ is invertible on $X_0^{\perp}$ with a bounded inverse,
see bound (\ref{resolvent-1} below.

Since $\mathcal{S} \mathcal{S} = \mathcal{I}$, the range of $\mathcal{S} \mathcal{H}$ is orthogonal
with respect to the generalized null space $Y_0 \subset L^2(\mathbb{R})$ of the adjoint operator
$\mathcal{H} \mathcal{S}$, which is spanned  by the vectors in $\mathcal{S} \Phi_V$.
Furthermore, we can apply the partition of $\Phi_V$ as
$\Phi_V^{(0)} = [{\bf V}_t, {\bf V}_g]$ and $\Phi_V^{(1)} = [\tilde{\bf V}_t, \tilde{\bf V}_g]$.
Given the computations above, we consider the Lyapunov--Schmidt decomposition in the form
\begin{equation}\label{eqMTM_decomp}
\left\{ \begin{array}{l}
\lambda = p (\Lambda + \mu_p), \\
{\bf V} = \Phi_V^{(0)} \vec{\alpha}_p + p \Phi_V^{(1)} ((\Lambda + \mu_p) \vec{\alpha}_p + \vec{\gamma}_p)+{\bf V}_p, \end{array} \right.
\end{equation}
where $\Lambda \in \mathbb{C}$ is $p$-independent, whereas $\mu_p \in \mathbb{C}$,
$\vec{\alpha}_p \in \mathbb{C}^2$, $\vec{\gamma}_p \in \mathbb{C}^2$, and ${\bf V}_p \in H^1(\mathbb{R})$ may depend on $p$.
For uniqueness of the decomposition, we use the Fredholm theory and require that
the correction term ${\bf V}_p \in H^1(\mathbb{R}) \cap Y_0^{\perp}$ satisfy the orthogonality conditions:
\begin{equation}
\label{Vp_constraint}
\langle \Phi_V, \mathcal{S} {\bf V}_p\rangle_{L^2} = 0.
\end{equation}
Substituting expansions \eqref{eqMTM_decomp} into the spectral problem \eqref{eqMTM}, we obtain
\begin{eqnarray}\nonumber
& \phantom{t} &
\left( \mathcal{H} + p^2 \mathcal{I} - i p (\Lambda + \mu_p) \mathcal{S} \right) {\bf V}_p +
p^2 \left( \Phi_V^{(0)} \vec{\alpha}_p + p \Phi_V^{(1)} ((\Lambda + \mu_p) \vec{\alpha}_p + \vec{\gamma}_p)\right) \\
\label{eqMTM_Vp} & \phantom{t} & \phantom{text} =
i p^2 (\Lambda + \mu_p) \mathcal{S} \Phi_V^{(1)} ((\Lambda + \mu_p)\vec{\alpha}_p + \vec{\gamma}_p)
- i p \mathcal{S} \Phi_V^{(0)} \vec{\gamma}_p.
\end{eqnarray}

Since $p^2 \mathcal{I}$ is a bounded self-adjoint perturbation to $\mathcal{H}$, there exist
positive constants $\lambda_0$, $p_0$, and $C_0$ such that for all $|\lambda| < \lambda_0$, $|p| < p_0$,
and all ${\bf f} \in X_0^{\perp} \subset L^2(\mathbb{R})$, there exists a unique $\mathcal{A}_{\lambda,p}^{-1} {\bf f}
\in Y_0^{\perp}$ satisfying
\begin{equation}
\label{resolvent-1}
\| \mathcal{A}_{\lambda,p}^{-1} {\bf f} \|_{L^2} \leq C_0 \| {\bf f} \|_{L^2}.
\end{equation}
Moreover, $\mathcal{A}_{\lambda,p}^{-1} {\bf f} \in H^1(\mathbb{R})$. Therefore, in order to solve
equation (\ref{eqMTM_Vp}) for ${\bf V}_p$ in $H^1(\mathbb{R}) \cap Y_0^{\perp}$,
we project the equation to $X_0^{\perp}$.
It makes sense to do so separately for $\Phi_V^{(0)}$ and $\Phi_V^{(1)}$.

Using the projection matrices (\ref{projections-matrix-2}) and (\ref{projections-matrix-3})
as well as the orthogonality conditions (\ref{Vp_constraint}), we obtain
\begin{eqnarray} \label{Eq1}
p^2 \langle \Phi_V^{(0)}, {\bf V}_p \rangle_{L^2} + p^2
\langle \Phi_V^{(0)}, \Phi_V^{(0)} \rangle_{L^2} \vec{\alpha}_p
=  i p^2 (\Lambda + \mu_p) \langle \Phi_V^{(0)}, \mathcal{S} \Phi_V^{(1)} \rangle_{L^2}
((\Lambda + \mu_p) \vec{\alpha}_p + \vec{\gamma}_p)
\end{eqnarray}
and
\begin{eqnarray} \label{Eq2}
p^2 \langle \Phi_V^{(1)}, {\bf V}_p \rangle_{L^2} + p^3
\langle \Phi_V^{(1)}, \Phi_V^{(1)} \rangle_{L^2} ((\Lambda + \mu_p) \vec{\alpha}_p + \vec{\gamma}_p )
=  -i p \langle \Phi_V^{(1)}, \mathcal{S} \Phi_V^{(0)} \rangle_{L^2} \vec{\gamma}_p.
\end{eqnarray}

The resolvent estimate (\ref{resolvent-1}) and the inverse function theorem imply
that,  under  the constraints (\ref{Eq1}) and (\ref{Eq2}),
there are positive numbers $p_1 \leq p_0$, $\mu_1$, and $C_1$ such that
for every $|p| < p_1$ and  $|\mu_p| < \mu_1$, there exists a unique solution of equation
\eqref{eqMTM_Vp} for ${\bf V}_p$ in $H^1(\mathbb{R}) \cap Y_0^{\perp}$ satisfying the estimate
\begin{equation} \label{Vp}
\| {\bf V}_p \|_{L^2} \leq C_1 \left( p^2  \| \vec{\alpha}_p \| + |p| \| \vec{\gamma}_p \|) \right).
\end{equation}
Substituting this solution to the projection equations (\ref{Eq1}) and (\ref{Eq2}),
we shall be looking for values of $\Lambda$, $\mu_p$, $\vec{\alpha}_p$, and $\vec{\gamma}_p$
for $|p| < p_1$ sufficiently small. Using the estimate (\ref{Vp}), we realize that
the leading order of the $2$-by-$2$ matrix equation (\ref{Eq1}) is
\begin{eqnarray} \label{Eq3}
\langle \Phi_V^{(0)}, \Phi_V^{(0)} \rangle_{L^2} \vec{c}
=  i \Lambda^2 \langle \Phi_V^{(0)}, \mathcal{S} \Phi_V^{(1)} \rangle_{L^2} \vec{c}, \quad \vec{c} \in \mathbb{C}^2.
\end{eqnarray}
This equation is diagonal and admits two eigenvalues for $\Lambda^2$ given by
$\Lambda_r(\omega)^2$ and $-\Lambda_i(\omega)^2$. Choosing $\Lambda^2$ being equal to
one of the two eigenvalues (which are distinct), we obtain a rank-one coefficient matrix
for the $2$-by-$2$ matrix equation (\ref{Eq1}) at the leading order.

For simplicity, let us choose $\Lambda^2 = \Lambda_r(\omega)^2$ (the other case is
considered similarly) and represent $\vec{\alpha}_p = (\alpha_p,\beta_p)^t$ and
$\vec{\gamma}_p = (\gamma_p,\delta_p)^t$. In this case, $\alpha_p$ can be normalized to
unity independently of $p$, after which the $2$-by-$2$ matrix equation (\ref{Eq1}) divided by $p^2$
is rewritten in the following explicit form
\begin{eqnarray} \label{Eq4}
\left[ \begin{array}{cc} \| {\bf V}_t \|_{L^2}^2 & 0 \\
0 & \| {\bf V}_g \|_{L^2}^2 \end{array} \right]
\left[ \begin{array}{l}\left(1 + \frac{\mu_p}{\Lambda_r}\right)^2 - 1 +
\frac{\Lambda_r + \mu_p}{\Lambda_r^2} \gamma_p \\
-\frac{\Lambda_r^2}{\Lambda_i^2} \left(1 + \frac{\mu_p}{\Lambda_r}\right)^2 \beta_p
- \beta_p - \frac{\Lambda_r + \mu_p}{\Lambda_i^2}  \delta_p  \end{array} \right] =  \langle \Phi_V^{(0)}, {\bf V}_p \rangle_{L^2}.
\end{eqnarray}

We invoke the implicit function theorem for vector functions.
It follows from the estimate (\ref{Vp}) that there are positive numbers $p_2 \leq p_1$ and $C_2$ such that
for every $|p| < p_2$, there exists a unique solution of
the $2$-by-$2$ matrix equation (\ref{Eq4}) for $\mu_p$ and $\beta_p$ satisfying the estimate
\begin{equation} \label{mu}
|\mu_p| + |\beta_p| \leq C_2 \left( \| \vec{\gamma}_p \| + \| {\bf V}_p \|_{L^2} \right) \leq
C_2 \left( \| \vec{\gamma}_p \| + p^2 \right),
\end{equation}
where the last inequality with a modified value of constant $C_2$ is due to the estimate (\ref{Vp}).

Finally, we divide the $2$-by-$2$ matrix equation (\ref{Eq2}) by $p$ and
rewrite it in the form
\begin{eqnarray} \label{Eq5}
-i \langle \Phi_V^{(1)}, \mathcal{S} \Phi_V^{(0)} \rangle_{L^2} \vec{\gamma}_p =
p \langle \Phi_V^{(1)}, {\bf V}_p \rangle_{L^2} + p^2
\langle \Phi_V^{(1)}, \Phi_V^{(1)} \rangle_{L^2} ((\Lambda + \mu_p) \vec{\alpha}_p + \vec{\gamma}_p ).
\end{eqnarray}
Thanks to the estimates (\ref{Vp}) and (\ref{mu}), the matrix equation (\ref{Eq5})
can be solved for $\vec{\gamma}_p$ by the implicit function theorem, if $p$ is sufficiently small
and $\vec{V}_p$, $\mu_p$, and $\vec{\alpha}_p$ are substituted from solutions of the previous equations.
As a result, there are positive numbers $p_3 \leq p_2$ and $C_3$ such that
for every $|p| < p_3$, there exists a unique solution of
the $2$-by-$2$ matrix equation (\ref{Eq5}) for $\vec{\gamma}_p$ satisfying the estimate
\begin{equation} \label{gamma}
\| \vec{\gamma}_p \| \leq C_3 \left( p^2 + p \| {\bf V}_p \|_{L^2} \right) \leq C_3 p^2,
\end{equation}
where the last inequality with a modified value of constant $C_3$ is due to the estimate (\ref{Vp}).

Decomposition (\ref{eqMTM_decomp}) and estimates (\ref{Vp}), (\ref{mu}), and (\ref{gamma})
justify the asymptotic expansion (\ref{asymptMTM-1}).
It remains to prove that the eigenvalue $\lambda = p(\Lambda_r + \mu_p)$ is purely real.
Since $\Lambda_r$ is real, the result holds if $\mu_p$ is real. Assume that $\mu_p$ has a
nonzero imaginary part. By Proposition \ref{proposition-symmetry}, there exists another
distinct eigenvalue of the spectral problem (\ref{eqMTM}) given by $\lambda = (p\Lambda_r + \bar{\mu}_p)$
such that $\bar{\mu}_p = \mathcal{O}(p^2)$ as $p \to 0$. However, the existence of this
distinct eigenvalue
contradicts the uniqueness of constructing of $\mu_p$ and all terms in the decomposition (\ref{eqMTM_decomp}).
Therefore, $\bar{\mu}_p = \mu_p$, so that $\lambda = p(\Lambda_r + \mu_p)$ is real.

The asymptotic expansion (\ref{asymptMTM-2}) is
proved similarly with the normalization $\beta_p = 1$ and
the choice $\Lambda^2 = -\Lambda_i(\omega)^2$ among eigenvalues of
the reduced eigenvalue problem (\ref{Eq3}).
\end{proof1}

\subsection{Perturbation theory for the massive Gross--Neveu model}

In the case of the massive Gross--Neveu model with (\ref{matrix-3-4}) and (\ref{matrix-4}),
the block-diagonalized system (\ref{linear2}) can be rewritten in the explicit form
\begin{equation} \label{eqH}
\left( \begin{matrix} H_+ & 0 \\ 0 & H_- \end{matrix} \right) {\bf V}
+ i p \left( \begin{matrix} 0 & \sigma_1 \\ -\sigma_1 & 0 \end{matrix} \right) {\bf V} =
i \lambda \left( \begin{matrix} 0 & \sigma_3 \\ \sigma_3 & 0 \end{matrix} \right) {\bf V},
\end{equation}
where $\sigma_1$ and $\sigma_3$ are the Pauli matrices, whereas
$$
H_+ = \left(\begin{matrix}-i\partial_x+\omega + 2 |U_{\omega}|^2&-1 + U_{\omega}^2 + 3\overline{U}_{\omega}^2\\
-1 + \overline{U}_{\omega}^2 + 3 U_{\omega}^2 & i\partial_x+\omega + 2|U_{\omega}|^2 \end{matrix}\right) \quad
\mbox{\rm and} \quad
H_- = \left(\begin{matrix}-i\partial_x+\omega & 1 - U_{\omega}^2 - \overline{U}_{\omega}^2 \\
1-U_{\omega}^2-\overline{U}_{\omega}^2  &i\partial_x+\omega \end{matrix}\right).
$$
We note again the symmetry relation (\ref{symmetry-Dirac-1}), which applies to the Dirac operators
$H_{\pm}$ for the massive Gross--Neveu model as well. From this symmetry, we derive the result,
which is similar to Proposition \ref{proposition-symmetry} and is proved directly.

\begin{prop}
\label{proposition-symmetry-GN}
If $\lambda$ is an eigenvalue of the spectral problem (\ref{eqH}) with $p \in \mathbb{R}$ and
the eigenvector ${\bf V} = (v_1,v_2,v_3,v_4)^t$, then $-\bar{\lambda}$ is also an eigenvalue
of the same problem with the eigenvector $(\bar{v}_2,\bar{v}_1,-\bar{v}_4,-\bar{v}_3)^t$,
whereas $\bar{\lambda}$ and $-\lambda$ are eigenvalues of the spectral problem (\ref{eqH}) with $-p \in \mathbb{R}$
and the eigenvectors $(\bar{v}_2,\bar{v}_1,\bar{v}_4,\bar{v}_3)^t$ and $(v_1,v_2,-v_3,-v_4)^t$, respectively.
Consequently, for every $p \in \mathbb{R}$, eigenvalues $\lambda$ of the spectral problem (\ref{eqH})
are symmetric about the imaginary axis.
\end{prop}

For the sake of simplicity, we use again the notations
$$
\mathcal{H} = \left( \begin{matrix} H_+ & 0 \\ 0 & H_- \end{matrix} \right), \quad
\mathcal{P} = i \left( \begin{matrix} 0 & \sigma_1 \\ -\sigma_1 & 0 \end{matrix} \right), \quad
\mathcal{S} = \left( \begin{matrix} 0 & \sigma_3 \\ \sigma_3 & 0 \end{matrix} \right).
$$
Besides the eigenvectors (\ref{kernel-1}) and the generalized eigenvectors (\ref{kernel-2}), we need solutions
of the linear inhomogeneous equations
\begin{equation}\label{kernel-3-systems}
\mathcal{H} {\bf V} = - \mathcal{P} {\bf V}_{t,g},
\end{equation}
which are given by
\begin{equation}\label{kernel-3}
\check{\mathbf{V}}_t = -\frac{1}{2}\left(\begin{matrix} 0 \\ 0 \\ \overline{U}_{\omega} \\ -U_{\omega}\end{matrix}\right) \quad \mbox{and} \quad
\check{\mathbf{V}}_g = -\frac{1}{2\omega}\left(\begin{matrix} \overline{U}_{\omega} \\ -U_{\omega} \\ 0 \\ 0 \end{matrix}\right).
\end{equation}
The existence of these explicit expressions is checked by elementary substitution.

We apply again the partition of $\Phi_V$ as
$\Phi_V^{(0)} = [{\bf V}_t, {\bf V}_g]$ and $\Phi_V^{(1)} = [\tilde{\bf V}_t, \tilde{\bf V}_g]$.
In addition, we augment the matrix $\Phi_V$ with $\Phi_V^{(2)} = [\check{\bf V}_t, \check{\bf V}_g]$
and compute the missing entries in the projection matrices:
\begin{equation}
\label{aug-projections-1}
\langle \Phi_V^{(0)}, \mathcal{S} \Phi_V^{(2)} \rangle_{L^2} = \langle \Phi_V^{(2)}, \mathcal{S} \Phi_V^{(2)} \rangle_{L^2} =
\left[ \begin{array}{cc} 0 & 0 \\ 0 & 0 \end{array} \right],
\end{equation}
and
\begin{equation}
\label{aug-projections-2}
\langle \Phi_V^{(1)}, \mathcal{S} \Phi_V^{(2)} \rangle_{L^2} = \left[ \begin{array}{cc} 0 & 0 \\
\langle \tilde{\bf V}_g, \mathcal{S} \check{\bf V}_t  \rangle_{L^2} & 0 \end{array} \right].
\end{equation}
Indeed, in addition to the matrix elements, which are trivially zero, we check that
\begin{equation}
\label{projections-matrix-a1}
\langle {\bf V}_{g}, \mathcal{S} \check{\bf V}_{g} \rangle_{L^2} = \frac{i}{2\omega}
\int_{\mathbb{R}} \left( \bar{U}^2_{\omega} - U_{\omega}^2 \right) dx = 0,
\end{equation}
because ${\rm Im}(U_{\omega}^2)$ is an odd function of $x$, and
\begin{eqnarray}
\langle \tilde{\bf V}_{t}, \mathcal{S} \check{\bf V}_{g} \rangle_{L^2} = \frac{i}{2}
\int_{\mathbb{R}} x \left( \bar{U}^2_{\omega} - U_{\omega}^2 \right) dx + \frac{1}{4 \omega} \int_{\mathbb{R}} (\bar{U}_{\omega}^2 + U_{\omega}^2 ) dx = 0, \label{projections-matrix-a2}
\end{eqnarray}
where the exact expression (\ref{Soler-soliton}) is used. On the other hand, we have
\begin{eqnarray}
\nonumber
\langle \tilde{\bf V}_{g}, \mathcal{S} \check{\bf V}_{t} \rangle_{L^2} & = & -\frac{1}{4}
\frac{d}{d\omega} \int_{\mathbb{R}} \left( \bar{U}^2_{\omega} + U_{\omega}^2 \right) dx \\
& = & -\frac{1}{2} \frac{d}{d\omega} \log\left( \frac{1 + \omega + \sqrt{1-\omega^2}}{1 + \omega - \sqrt{1-\omega^2}} \right)
= \frac{1}{2 \omega \sqrt{1 - \omega^2}}, \label{projections-matrix-a3}
\end{eqnarray}
Similarly, we compute the zero projection matrices
\begin{equation}
\label{aug-projections-3}
\langle \Phi_V^{(0)}, \mathcal{P} \Phi_V^{(0)} \rangle_{L^2} =
\langle \Phi_V^{(0)}, \mathcal{P} \Phi_V^{(1)} \rangle_{L^2} =
\langle \Phi_V^{(1)}, \mathcal{P} \Phi_V^{(2)} \rangle_{L^2} =
\left[ \begin{array}{cc} 0 & 0 \\ 0 & 0 \end{array} \right]
\end{equation}
and the nonzero projection matrices
\begin{equation}
\label{aug-projections-4}
\langle \Phi_V^{(1)}, \mathcal{P} \Phi_V^{(1)} \rangle_{L^2} =
\left[ \begin{array}{cc} 0 & \langle \tilde{\bf V}_t, \mathcal{P} \tilde{\bf V}_g \rangle_{L^2} \\
\langle \tilde{\bf V}_g, \mathcal{P} \tilde{\bf V}_t \rangle_{L^2} & 0 \end{array} \right],
\end{equation}
\begin{equation}
\label{aug-projections-5}
\langle \Phi_V^{(0)}, \mathcal{P} \Phi_V^{(2)} \rangle_{L^2} =
\left[ \begin{array}{cc} \langle {\bf V}_t, \mathcal{P} \check{\bf V}_t \rangle_{L^2} & 0 \\
0 & \langle {\bf V}_g, \mathcal{P} \check{\bf V}_g \rangle_{L^2} \end{array} \right],
\end{equation}
and
\begin{equation}
\label{aug-projections-6}
\langle \Phi_V^{(2)}, \mathcal{P} \Phi_V^{(2)} \rangle_{L^2} =
\left[ \begin{array}{cc} 0 & \langle \check{\bf V}_t, \mathcal{P} \check{\bf V}_g \rangle_{L^2} \\
\langle \check{\bf V}_g, \mathcal{P} \check{\bf V}_t \rangle_{L^2} & 0 \end{array} \right].
\end{equation}

Indeed, the first matrix in (\ref{aug-projections-3}) is zero because the Fredholm conditions
for the inhomogeneous linear systems (\ref{kernel-3-systems}) are satisfied. The second matrix in
(\ref{aug-projections-3}) is zero because
\begin{equation}
\label{projections-matrix-a4}
\langle {\bf V}_{t}, \mathcal{P} \tilde{\bf V}_{t} \rangle_{L^2} = \frac{\omega}{2}
\int_{\mathbb{R}} \left( U_{\omega}^2 - \bar{U}^2_{\omega} \right) dx = 0
\end{equation}
and
\begin{equation}
\label{projections-matrix-a5}
\langle {\bf V}_{g}, \mathcal{P} \tilde{\bf V}_{g} \rangle_{L^2} = \frac{1}{2} \frac{d}{d \omega}
\int_{\mathbb{R}} \left( U_{\omega}^2 - \bar{U}^2_{\omega} \right) dx = 0.
\end{equation}
The third matrix in (\ref{aug-projections-3}) is zero because
\begin{equation}
\label{projections-matrix-a4-1}
\langle \tilde{\bf V}_{t}, \mathcal{P} \check{\bf V}_{g} \rangle_{L^2} = -
\int_{\mathbb{R}} x |U_{\omega}|^2 dx = 0
\end{equation}
and
\begin{equation}
\label{projections-matrix-a4-2}
\langle \tilde{\bf V}_{g}, \mathcal{P} \check{\bf V}_{t} \rangle_{L^2} = \frac{i}{2}
\int_{\mathbb{R}}\left( U_{\omega} \partial_{\omega} \bar{U}_{\omega} - \bar{U}_{\omega} \partial_{\omega} U_{\omega} \right) dx = 0.
\end{equation}

For the projection matrices (\ref{aug-projections-4}), (\ref{aug-projections-5}),
and (\ref{aug-projections-6}), we compute the nonzero elements explicitly:
\begin{eqnarray}
\label{projections-matrix-a6}
\langle \tilde{\bf V}_{t}, \mathcal{P} \tilde{\bf V}_{g} \rangle_{L^2} & = & \frac{i}{4} \frac{d}{d \omega}
\int_{\mathbb{R}} \left( U_{\omega}^2 + \bar{U}_{\omega}^2 \right) dx +
\frac{\omega}{2} \frac{d}{d \omega}
\int_{\mathbb{R}} x \left( U_{\omega}^2 - \bar{U}_{\omega}^2 \right) dx \neq 0, \\
\label{projections-matrix-a7}
\langle {\bf V}_{t}, \mathcal{P} \check{\bf V}_{t} \rangle_{L^2} & = & \frac{i}{2}
\int_{\mathbb{R}} \left( U_{\omega} \bar{U}_{\omega}' - \bar{U}_{\omega} U_{\omega}' \right) dx \neq 0,\\
\label{projections-matrix-a8}
\langle {\bf V}_{g}, \mathcal{P} \check{\bf V}_{g} \rangle_{L^2} & = & -\frac{1}{\omega}
\int_{\mathbb{R}} |U_{\omega}|^2 dx \neq 0, \\
\label{projections-matrix-a9}
\langle \check{\bf V}_{t}, \mathcal{P} \check{\bf V}_{g} \rangle_{L^2} & = & \frac{i}{4 \omega}
\int_{\mathbb{R}} \left( U_{\omega}^2 + \bar{U}_{\omega}^2 \right) dx \neq 0.
\end{eqnarray}

The following result gives the outcome of the perturbation theory
associated with the generalized null space of the spectral stability
problem (\ref{eqH}). The result is equivalent to the part of Theorem \ref{specThm}
corresponding to the massive Gross--Neveu model. The asymptotic expressions
$\Lambda_r$ and $\Lambda_i$ for the corresponding eigenvalues $\lambda$ at the leading order in $p$
versus parameter $\omega$ are shown on Fig. \ref{FigTheoryB}.

\begin{lem}\label{lemmaSoler}
For every $\omega \in (0,1)$, there exists $p_0 > 0$ such that
for every $p$ with $0 < |p| < p_0$, the spectral stability problem \eqref{eqH}
admits a pair of purely imaginary eigenvalues $\lambda$ with the eigenvectors
${\bf V} \in H^1(\mathbb{R})$ such that
\begin{equation}
\label{asymptSoler-1}
\lambda = \pm i p \Lambda_i(\omega) + \mathcal{O}(p^3), \quad
{\bf V} = \mathbf{V}_t \pm i p \Lambda_i(\omega) \tilde{\mathbf{V}}_t + p \check{\mathbf{V}}_t + p \beta {\bf V}_g
+ \mathcal{O}_{H^1}(p^2) \quad \mbox{\rm as} \quad p \to 0,
\end{equation}
where $\Lambda_i(\omega) = \sqrt{\frac{I(\omega)}{1 + I(\omega)}} > 0$ with $I(\omega) > 0$
given by the explicit expression (\ref{Lambda-i-explicit}) below and $\beta$ is uniquely defined
in \eqref{beta} below.

Simultaneously, the spectral stability problem \eqref{eqH} admits a pair of eigenvalues $\lambda$ with ${\rm Re}(\lambda) \neq 0$
symmetric about the imaginary axis, and the eigenvector ${\bf V} \in H^1(\mathbb{R})$ such that
\begin{equation}
\label{asymptSoler-2}
\lambda = \pm p \Lambda_r(\omega) + \mathcal{O}(p^3), \quad
{\bf V} = \mathbf{V}_g \pm p \Lambda_r(\omega)  \tilde{\mathbf{V}}_g + p \check{\mathbf{V}}_g + p \alpha {\bf V}_t
+ \mathcal{O}_{H^1}(p^2) \quad \mbox{\rm as} \quad p \to 0,
\end{equation}
where $\Lambda_r = (1-\omega^2)^{1/2} > 0$ and $\alpha$ is uniquely defined in \eqref{alpha} below.
\end{lem}

We proceed with formal expansions, which are similar to the expansions \eqref{asymptMTM-3}.
However, because the $\mathcal{O}(p)$ terms appear explicitly in the spectral stability
problem (\ref{eqH}), we introduce the modified expansions as follows,
\begin{equation}
\label{asymptSoler}
\lambda = p \Lambda_1 + p^2 \Lambda_2 + \mathcal{O}(p^3), \quad
{\bf V} = {\bf V}_0 + p (\Lambda_1 {\bf V}_1+\check{{\bf V}}_1 + {\bf V}_0') + p^2 {\bf V}_2 + \mathcal{O}_{H^1}(p^3),
\end{equation}
where ${\bf V}_0$ and ${\bf V}_0'$ are spanned independently by the eigenvectors (\ref{kernel-1}), ${\bf V}_1$ is
spanned by the generalized eigenvectors (\ref{kernel-2}),
$\check{{\bf V}}_1$ is spanned by the vectors \eqref{kernel-3}, and ${\bf V}_2$ satisfies
the linear inhomogeneous equation
\begin{equation}
\label{linearSoler}
\mathcal{H} {\bf V}_2 = (i \Lambda_1 \mathcal{S} - \mathcal{P}) (\Lambda_1{\bf V}_1+\check{{\bf V}}_1 + {\bf V}'_0)
+ i \Lambda_2 \mathcal{S} {\bf V}_0.
\end{equation}
By the Fredholm alternative, there exists a solution ${\bf V}_2 \in H^1(\mathbb{R})$ of
the linear inhomogeneous equation (\ref{linearSoler}) if and only if
$\Lambda_1$ is found from the quadratic equation
\begin{equation}
\label{solvabilitySoler}
\langle {\bf W}_0, (i \Lambda_1 \mathcal{S} - \mathcal{P}) (\Lambda_1{\bf V}_1+\check{{\bf V}}_1 + {\bf V}_0')\rangle_{L^2} = 0,
\end{equation}
where ${\bf W}_0$ is again spanned  by the eigenvectors of $\mathcal{H}$ independently of ${\bf V}_0$.
Similar to the example of the massive Thirring model, the matrix eigenvalue problem (\ref{solvabilitySoler})
is diagonal with respect to the translational and gauge symmetries. As a result, subsequent computations
can be constructed independently for the two corresponding eigenvectors.

Selecting ${\bf V}_0 = {\bf W}_0 = {\bf V}_g$, ${\bf V}_1 = \tilde{\mathbf{V}}_g$, $\check{\bf V}_1 = \check{\bf V}_g$, and
${\bf V}_0' = \alpha {\bf V}_t$, we use (\ref{projections-matrix-2}), (\ref{nonzero-element-2}), (\ref{aug-projections-1}),
(\ref{aug-projections-3}), (\ref{aug-projections-5}), and (\ref{projections-matrix-a8})
in the solvability condition \eqref{solvabilitySoler}  and obtain the quadratic equation for $\Lambda_1$ in the explicit form
\begin{equation} \label{quadSoler1}
\Lambda_1^2 \frac{d}{d\omega}\int_{\mathbb{R}} |U_{\omega}|^2dx
+ \frac{1}{\omega} \int_{\mathbb{R}} |U_{\omega}|^2dx=0.
\end{equation}
Using the explicit expression (\ref{Soler-soliton}), we obtain
$$
\int_{\mathbb{R}} |U_{\omega}|^2dx = \frac{\sqrt{1-\omega^2}}{\omega}, \quad
\frac{d}{d\omega}\int_{\mathbb{R}} |U_{\omega}|^2dx = -\frac{1}{\omega^2 \sqrt{1 - \omega^2}}.
$$
As a result, we find from \eqref{quadSoler1} that $\Lambda_1^2 = 1- \omega^2 = \Lambda_r(\omega)^2$. Correction terms
$\Lambda_2$ and $\alpha$ are not determined up to this order of the asymptotic expansion.

Selecting now ${\bf V}_0 = {\bf W}_0 = {\bf V}_t$, ${\bf V}_1 = \tilde{\mathbf{V}}_t$, $\check{\bf V}_1 = \check{\bf V}_t$,
and ${\bf V}_0' = \beta {\bf V}_g$, we use (\ref{projections-matrix-2}), (\ref{nonzero-element-1}), (\ref{aug-projections-1}),
(\ref{aug-projections-3}), (\ref{aug-projections-5}), and (\ref{projections-matrix-a7})
in the solvability condition \eqref{solvabilitySoler}
and obtain the quadratic equation for $\Lambda_1$ in the explicit form
\begin{eqnarray}
\Lambda_1^2 \int_{\mathbb{R}} \left[  \omega |U_{\omega}|^2 + \frac{i}{2} (\bar{U}_{\omega} U_{\omega}' - U_{\omega} \bar{U}_{\omega}' ) \right] dx
+ \frac{i}{2} \int_{\mathbb{R}} (\bar{U}_{\omega} U_{\omega}' - U_{\omega} \bar{U}_{\omega}' ) dx = 0. \label{quadSoler2}
\end{eqnarray}
Expressing
$$
\frac{i}{2} \int_{\mathbb{R}} (\bar{U}_{\omega} U_{\omega}' - U_{\omega} \bar{U}_{\omega}' ) dx =
\int_{\mathbb{R}} \frac{(1-\omega^2)^2}{(1 + \omega \cosh(2 \mu x))^2} dx = \sqrt{1-\omega^2} I(\omega),
$$
where
\begin{equation}
\label{Lambda-i-explicit}
I(\omega) := (1 - \omega^2) \int_{0}^{\infty} \frac{dz}{(1 + \omega \cosh(z))^2}=1-\frac{1}{\sqrt{1-\omega^2}}\log\left(\frac{1-\sqrt{1-\omega^2}}{\omega}\right) > 0,
\end{equation}
we obtain from \eqref{quadSoler2} that $\Lambda_1^2 = -\frac{I(\omega)}{1 + I(\omega)} = -\Lambda_i(\omega)^2$.
Again, correction terms
$\Lambda_2$ and $\beta$ are not determined up to this order of the asymptotic expansion.

Justification of the formal expansion (\ref{asymptSoler}) and the proof of Lemma \ref{lemmaSoler}
is achieved by exactly the same argument as in the proof of Lemma \ref{lemmaMTM}. The proof relies
on the resolvent estimate (\ref{resolvent-1}), which is valid for the massive Gross--Neveu model,
because by Propositions \ref{proposition-continuous-spectrum} and \ref{proposition-kernel},
the zero eigenvalue of the operator $\mathcal{H} - i \lambda \mathcal{S}$ (which has algebraic multiplicity four) is isolated
from the rest of the spectrum.

Persistence of eigenvalues is proved with the symmetry in Proposition \ref{proposition-symmetry-GN}.
If an eigenvalue is expressed as $\lambda = p (i \Lambda_i(\omega) + \mu_p)$ with unique $\mu_p = \mathcal{O}(p)$ and
$\Lambda_i(\omega) > 0$, then nonzero real part of $\mu_p$ would contradict the symmetry of eigenvalues
about the imaginary axis. Therefore, ${\rm Re}(\mu_p) = 0$ and the eigenvalues in the expansion
(\ref{asymptSoler-1}) remain on the imaginary axis.
On the other hand, if another eigenvalue is expressed as
$\lambda = p (\Lambda_r(\omega) + \mu_p)$ with unique $\mu_p = \mathcal{O}(p)$ and
$\Lambda_r(\omega) > 0$, then $\mu_p$ may have in general a nonzero imaginary part, as it does not
contradict the symmetry of Proposition \ref{proposition-symmetry-GN} for a fixed $p \neq 0$. This is why
the statement of Lemma \ref{lemmaSoler} does not guarantee that the corresponding eigenvalues
in the expansion (\ref{asymptSoler-2}) are purely real.

In the end of this section, we will show that $\mu_p = \mathcal{O}(p^2)$, which justifies
the $\mathcal{O}(p^3)$ bound for the eigenvalues in the asymptotic expansions (\ref{asymptSoler-1}) and (\ref{asymptSoler-2}).
In this procedure, we will uniquely determine the parameters $\beta$ and $\alpha$
in the same asymptotic expansions.
Extending the expansion (\ref{asymptSoler}) to $p^3 \Lambda_3$ and $p^3 {\bf V}_3$ terms, we obtain
the linear inhomogeneous equation
\begin{equation}
\label{linearSoler-higher}
\mathcal{H} {\bf V}_3 = (i \Lambda_1 \mathcal{S} - \mathcal{P}) {\bf V}_2
+ i \Lambda_2 \mathcal{S} (\Lambda_1{\bf V}_1+\check{{\bf V}}_1 + {\bf V}_0')
+ i \Lambda_3 \mathcal{S} {\bf V}_0.
\end{equation}
The Fredholm solvability condition
\begin{equation}
\label{solvabilitySoler-higher}
\langle {\bf W}_0, (i \Lambda_1 \mathcal{S} - \mathcal{P}) {\bf V}_2 + i \Lambda_2 \mathcal{S}
(\Lambda_1{\bf V}_1+\check{{\bf V}}_1 + {\bf V}_0')\rangle_{L^2} = 0
\end{equation}
determines the correction terms $\Lambda_2$, $\beta$, and $\alpha$ uniquely.
Indeed, using (\ref{projections-matrix-2}) and (\ref{aug-projections-1}),
we rewrite the solvability condition (\ref{solvabilitySoler-higher}) in the form
\begin{eqnarray*}
i \langle {\bf W}_0, \mathcal{S} {\bf V}_1 \rangle_{L^2} \Lambda_2 \Lambda_1 & = & -
\langle {\bf W}_0, (i \Lambda_1 \mathcal{S} - \mathcal{P}) {\bf V}_2 \rangle_{L^2} \\
& = & -\langle (-i \bar{\Lambda}_1 \mathcal{S} - \mathcal{P}) {\bf W}_0,  {\bf V}_2 \rangle_{L^2} \\
& = & -\langle \mathcal{H} (-\bar{\Lambda}_1 {\bf W}_1 + \check{{\bf W}}_1),  {\bf V}_2 \rangle_{L^2} \\
& = & -\langle (-\bar{\Lambda}_1 {\bf W}_1 + \check{{\bf W}}_1), \mathcal{H}  {\bf V}_2 \rangle_{L^2} \\
& = & -\langle (-\bar{\Lambda}_1 {\bf W}_1 + \check{{\bf W}}_1), i \Lambda_2 \mathcal{S} {\bf V}_0 + (i \Lambda_1 \mathcal{S} - \mathcal{P})
(\Lambda_1 {\bf V}_1 + \check{\bf V}_1 + {\bf V}_0') \rangle_{L^2},
\end{eqnarray*}
where we have used the linear inhomogeneous equation (\ref{linearSoler})
and have introduced ${\bf W}_1$ and $\check{\bf W}_1$ from solutions
of the inhomogeneous equations $\mathcal{H} {\bf W}_1 = i \mathcal{S} {\bf W}_0$
and $\mathcal{H} \check{\bf W}_1 = - \mathcal{P} {\bf W}_0$.
Using
$$
\langle {\bf W}_1, i \mathcal{S} {\bf V}_0 \rangle_{L^2} = \langle {\bf W}_1, \mathcal{H} {\bf V}_1 \rangle_{L^2} =
\langle \mathcal{H} {\bf W}_1, {\bf V}_1 \rangle_{L^2} =
\langle i \mathcal{S} {\bf W}_0, {\bf V}_1 \rangle_{L^2} = -i \langle {\bf W}_0, \mathcal{S} {\bf V}_1 \rangle_{L^2}
$$
and
$$
\langle \check{\bf W}_1, i \mathcal{S} {\bf V}_0 \rangle_{L^2} = \langle \check{\bf W}_1, \mathcal{H} {\bf V}_1 \rangle_{L^2} =
\langle \mathcal{H} \check{\bf W}_1, {\bf V}_1 \rangle_{L^2} =
-\langle \mathcal{P} {\bf W}_0, {\bf V}_1 \rangle_{L^2} = -\langle {\bf W}_0, \mathcal{P} {\bf V}_1 \rangle_{L^2} = 0,
$$
where the last equality is due to (\ref{aug-projections-3}), we rewrite
the solvability equation in the form
\begin{eqnarray}
2 i \langle {\bf W}_0, \mathcal{S} {\bf V}_1 \rangle_{L^2} \Lambda_2 \Lambda_1 =
-\langle (-\bar{\Lambda}_1 {\bf W}_1 + \check{{\bf W}}_1), (i \Lambda_1 \mathcal{S} - \mathcal{P})
(\Lambda_1 {\bf V}_1 + \check{\bf V}_1 + {\bf V}_0') \rangle_{L^2}.\label{solv-eq-again-1}
\end{eqnarray}
Removing zero entries by using (\ref{projections-matrix-2}), (\ref{aug-projections-1}), and (\ref{aug-projections-3}),
we rewrite equation (\ref{solv-eq-again-1}) in the form
\begin{eqnarray}
\nonumber
2 i \langle {\bf W}_0, \mathcal{S} {\bf V}_1 \rangle_{L^2} \Lambda_2 \Lambda_1 & = &
\Lambda_1^2 \left( i \langle {\bf W}_1, \mathcal{S} {\bf V}_0' \rangle_{L^2}
+ i \langle {\bf W}_1, \mathcal{S} \check{\bf V}_1 \rangle_{L^2} - i \langle \check{\bf W}_1, \mathcal{S} {\bf V}_1 \rangle_{L^2}
- \langle {\bf W}_1, \mathcal{P} {\bf V}_1 \rangle_{L^2} \right) \\
& \phantom{t} & + \langle \check{\bf W}_1, \mathcal{P} \check{\bf V}_1 \rangle_{L^2}
+ \langle \check{\bf W}_1, \mathcal{P} {\bf V}_0' \rangle_{L^2}. \label{solv-eq-again-2}
\end{eqnarray}
We shall now write equation (\ref{solv-eq-again-2}) explicitly as the $2$-by-$2$ matrix equation by using
${\bf V}_0 = {\bf W}_0 = \Phi_V^{(0)}$, ${\bf V}_1 = {\bf W}_1 = \Phi_V^{(1)}$,
$\check{\bf V}_1 = \check{\bf W}_1 = \Phi_V^{(2)}$, and
$$
{\bf V}_0' = \Phi_V^{(0)} \left[ \begin{array}{cc} 0 & \alpha \\ \beta & 0 \end{array} \right] =
\left[ \beta {\bf V}_g, \alpha {\bf V}_t \right].
$$
Using (\ref{projections-matrix-2}), (\ref{aug-projections-2}), (\ref{aug-projections-4}), (\ref{aug-projections-5}),
and (\ref{aug-projections-6}), we rewrite equation (\ref{solv-eq-again-2}) in the matrix form
\begin{eqnarray}
\nonumber
2i \left[\begin{matrix} \langle {\bf V}_t, \mathcal{S} \tilde{\bf V}_t \rangle_{L^2} & 0 \\
0 & \langle {\bf V}_g, \mathcal{S} \tilde{\bf V}_g \rangle_{L^2} \end{matrix} \right] \Lambda_2\Lambda_1 & = &
i \Lambda_1^2
\left[ \begin{matrix} \langle \tilde{\bf V}_t, \mathcal{S} {\bf V}_t \rangle_{L^2} & 0 \\
0 & \langle \tilde{\bf V}_g, \mathcal{S} {\bf V}_g \rangle_{L^2} \end{matrix} \right]
\left[ \begin{array}{cc} 0 & \alpha \\ \beta & 0 \end{array} \right]\\
\nonumber & \phantom{t} &
+ \left[ \begin{matrix} \langle \check{\bf V}_t, \mathcal{P} {\bf V}_t \rangle_{L^2} & 0 \\
0 & \langle \check{\bf V}_g, \mathcal{P} {\bf V}_g \rangle_{L^2} \end{matrix} \right]
\left[ \begin{array}{cc} 0 & \alpha \\ \beta & 0 \end{array} \right] \\
\nonumber & \phantom{t} &
+ i \Lambda_1^2 \left[ \begin{array}{cc} 0 & -\langle \check{\bf V}_t, \mathcal{S} \tilde{\bf V}_g  \rangle_{L^2} \\
\langle \tilde{\bf V}_g, \mathcal{S} \check{\bf V}_t  \rangle_{L^2} & 0 \end{array} \right]\\
\nonumber & \phantom{t} &
- \Lambda_1^2 \left[ \begin{array}{cc} 0 & \langle \tilde{\bf V}_t, \mathcal{P} \tilde{\bf V}_g \rangle_{L^2} \\
\langle \tilde{\bf V}_g, \mathcal{P} \tilde{\bf V}_t \rangle_{L^2} & 0 \end{array} \right] \\
& \phantom{t} &
+ \left[ \begin{array}{cc} 0 & \langle \check{\bf V}_t, \mathcal{P} \check{\bf V}_g \rangle_{L^2} \\
\langle \check{\bf V}_g, \mathcal{P} \check{\bf V}_t \rangle_{L^2} & 0 \end{array} \right], \label{matrix-special}
\end{eqnarray}
where $\Lambda_1$ is defined uniquely from either solution of the quadratic equations (\ref{quadSoler1}) and (\ref{quadSoler2}).
Because the $2$-by-$2$ matrix on the right-hand side of equation (\ref{matrix-special})
is anti-diagonal, we obtain $\Lambda_2 = 0$ for every choice of $\Lambda_1$.

Now, we check that the coefficients $\alpha$ and $\beta$ are uniquely
determined from the right-hand side of the matrix equation (\ref{matrix-special}).
The coefficient $\alpha$ is determined for
$\Lambda_1^2 = \Lambda_r(\omega)^2 > 0$ from the anti-diagonal entry
$$
i \Lambda_1^2 \langle \tilde{\bf V}_t, \mathcal{S} {\bf V}_t \rangle_{L^2} +
\langle \check{\bf V}_t, \mathcal{P} {\bf V}_t \rangle_{L^2}
= i \langle \tilde{\bf V}_t, \mathcal{S} {\bf V}_t \rangle_{L^2} \left( \Lambda_r(\omega)^2 + \Lambda_i(\omega)^2\right) \neq 0
$$
Since the denominator is nonzero for every $\omega \in (0,1)$, we obtain
the unique expression for $\alpha$:
\begin{equation}\label{alpha}
\alpha = \frac{\Lambda_r(\omega)^2 \left( \langle \tilde{\bf V}_t, \mathcal{P} \tilde{\bf V}_g \rangle_{L^2}
+ i \langle \check{\bf V}_t, \mathcal{S} \tilde{\bf V}_g \rangle_{L^2} \right)
- \langle \check{\bf V}_t, \mathcal{P} \check{\bf V}_g \rangle_{L^2}}{i \langle \tilde{\bf V}_t, \mathcal{S} {\bf V}_t \rangle_{L^2}
\left( \Lambda_r(\omega)^2 + \Lambda_i(\omega)^2\right)}.
\end{equation}
Similarly, the coefficient $\beta$ is determined for
$\Lambda_1^2 = -\Lambda_i(\omega)^2 < 0$ from the anti-diagonal entry
$$
i \Lambda_1^2 \langle \tilde{\bf V}_g, \mathcal{S} {\bf V}_g \rangle_{L^2} +
\langle \check{\bf V}_g, \mathcal{P} {\bf V}_g \rangle_{L^2} = -i
\langle \tilde{\bf V}_g, \mathcal{S} {\bf V}_g \rangle_{L^2} \left(\Lambda_i(\omega)^2 +  \Lambda_r(\omega)^2 \right) \neq 0.
$$
Again the denominator is nonzero for every $\omega \in (0,1)$, so that we obtain
the unique expression for $\beta$:
\begin{equation}\label{beta}
\beta = \frac{\Lambda_i(\omega)^2 \left( i \langle \check{\bf V}_g, \mathcal{S} \tilde{\bf V}_t \rangle_{L^2}
- \langle \tilde{\bf V}_g, \mathcal{P} \tilde{\bf V}_t \rangle_{L^2} \right)
- \langle \check{\bf V}_g, \mathcal{P} \check{\bf V}_t \rangle_{L^2}}{-i \langle \tilde{\bf V}_g, \mathcal{S} {\bf V}_g \rangle_{L^2}
\left(\Lambda_i(\omega)^2 +  \Lambda_r(\omega)^2 \right)}.
\end{equation}
These computations justify the $\mathcal{O}(p^3)$ terms in the expansions (\ref{asymptSoler-1})
and (\ref{asymptSoler-2}) for the eigenvalues $\lambda$.

\section{Numerical approximations}

We approximate eigenvalues of the spectral stability problems (\ref{eqMTM}) and (\ref{eqH})
for the massive Thirring and Gross--Neveu models with the Chebyshev interpolation method.
This method has been already applied to the linearized Dirac system in one dimension in \cite{CP}.
The block diagonalized systems in \eqref{eqMTM} and \eqref{eqH} are discretized on
the grid points
$$
x_j=L \tanh^{-1}(z_j), \quad j=0,1,\ldots,N,
$$
where $z_j = \cos\left(\frac{j\pi}{N}\right)$ is the Chebyshev node
and a scaling parameter $L$ is chosen suitably
so that the grid points are concentrated in the region, where the soliton
$U_{\omega}$ changes fast.  Note that $x_0 = \infty$ and $x_N = -\infty$.

According to the standard Chebyshev interpolation method \cite{Trefethen},
the first derivative that appears in the systems \eqref{eqMTM} and \eqref{eqH} is constructed
from the scaled Chebyshev differentiation matrix $\widetilde{D}_N$ of the size $(N+1)\times (N+1)$,
whose each element at $i^{\rm th}$ row and $j^{\rm th}$ column is given by 
$$
[\widetilde{D}_N]_{ij}=\frac{1}{L}\sech^2\left(\frac{x_i}{L} \right)[D_N]_{ij},
$$
where $D_N$ is the standard Chebyshev differentiation matrix (see page $53$ of \cite{Trefethen}) 
and the chain rule $\frac{d u}{dx} = \frac{dz}{dx}\frac{du}{dz}$ has been used. 
Denoting $I_N$ as an identity matrix of the size $(N+1)\times (N+1)$, we replace each term 
in the systems \eqref{eqMTM} and \eqref{eqH} as follows:
$$
\partial_x \rightarrow \widetilde{D}_N, \quad 1\rightarrow I_{N}, \quad
U_{\omega} \rightarrow \mbox{diag}(U_{\omega}(x_0), U_{\omega}(x_1),\cdots, U_{\omega}(x_N)),
$$
Due to the decay of the soliton $U_{\omega}$ to zero at infinity, we have $U_{\omega}(x_0) = U_{\omega}(x_N) = 0$.

The resulting discretized systems from \eqref{eqMTM} and \eqref{eqH} are of the size $4(N+1)\times 4(N+1)$.
Boundary conditions are naturally built into this formulation, because the elements
of the first and last rows of the matrix $[\widetilde{D}_N]_{ij}$ are zero. As a result,
eigenvalues from the first and last rows of the linear discretized system
are nothing but the end points of the continuous spectrum in Proposition \ref{proposition-continuous-spectrum},
whereas the boundary values of the vector ${\bf V}$ at the end points $x_0$ and $x_N$
are identically zero for all other eigenvalues of the linear discretized system.

Throughout all our numerical results, we pick the value of a scaling parameter $L$ to be $L=10$.
This choice ensures that the soliton solutions $U_{\omega}$ for all values of $\omega$ used in
our numerical experiments remain nonzero up to $16$ decimals on all interior grid points
$x_j$ with $1 \leq j \leq N-1$.

Now we begin explaining our numerical results.
Figure \ref{MTMFig} shows eigenvalues of the spectral stability problem
(\ref{eqMTM}) for the solitary wave of the massive Thirring model.
We set $\omega=0$ and display eigenvalues $\lambda$ in the complex plane for different values of $p$.
The subfigure at $p=0.2$ demonstrates our analytical result in Lemma \ref{lemmaMTM},
which predicts splitting of the zero eigenvalue of algebraic multiplicity four
into two pairs of real and imaginary eigenvalues. Increasing the value of $p$ further,
we observe emergence of imaginary eigenvalues from the edges of the continuous spectrum
branches, as seen at $p=0.32$. A pair of imaginary eigenvalues coalesces and bifurcates
into the complex plane with nonzero real parts, as seen at $p=0.36$, and later
absorbs back into the continuous spectrum branches, seen in the next subfigures.
We can also see emergence of a pair of imaginary eigenvalues from the edges of the continuous spectrum
branches at $p=0.915$. The pair bifurcates along the real axis after coalescence at the origin,
as seen at $p=1$. The gap of the continuous spectrum closes up at $p = 1$.
For a larger value of $p$, two pairs of real eigenvalues are seen to approach each other.

\begin{figure}[htbp]
        \centering
        \begin{subfigure}[htbp]{0.35\textwidth}
               \includegraphics[scale=0.55]{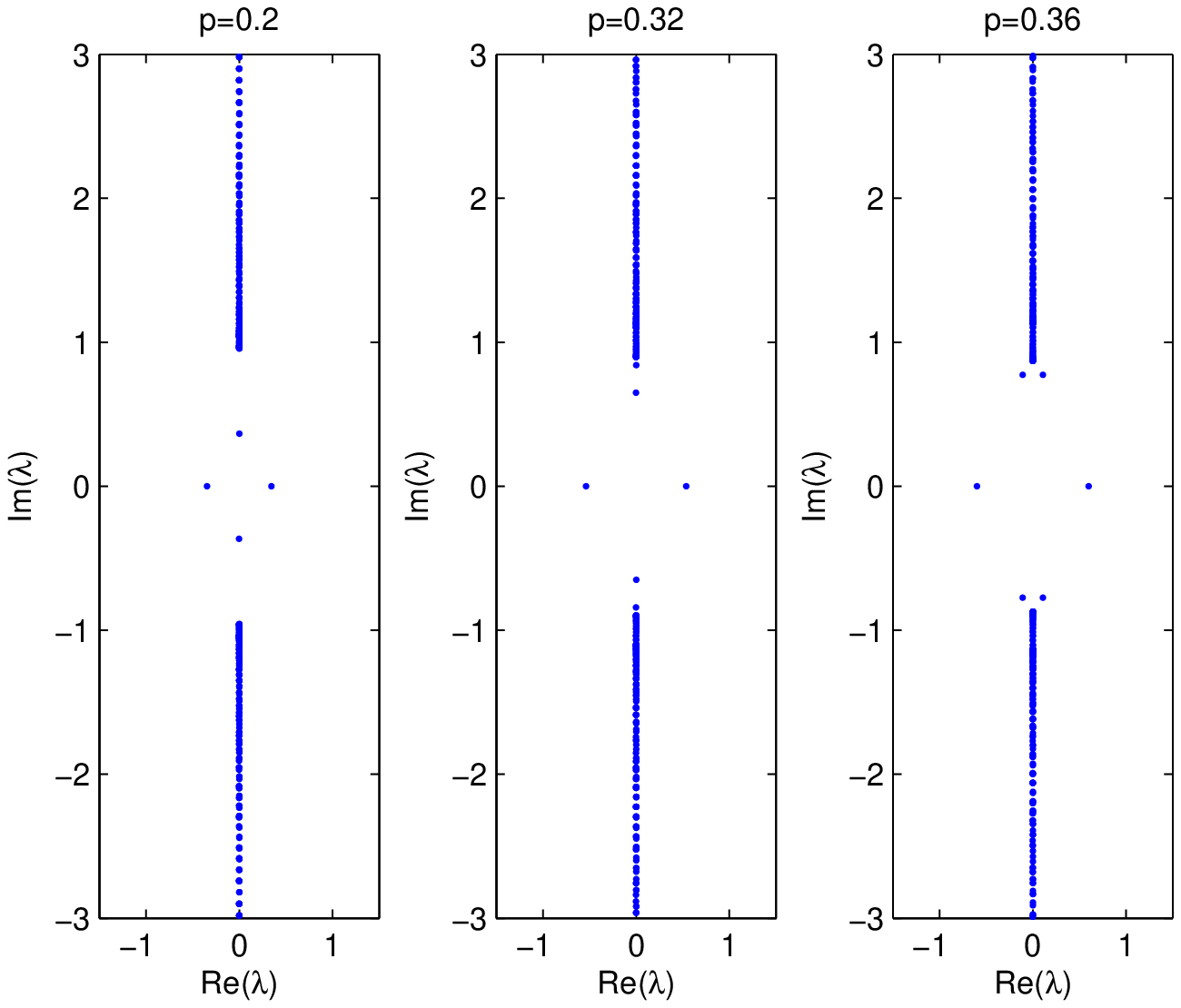}
        \end{subfigure}

        \begin{subfigure}[htbp]{0.35\textwidth}
              \includegraphics[scale=0.55]{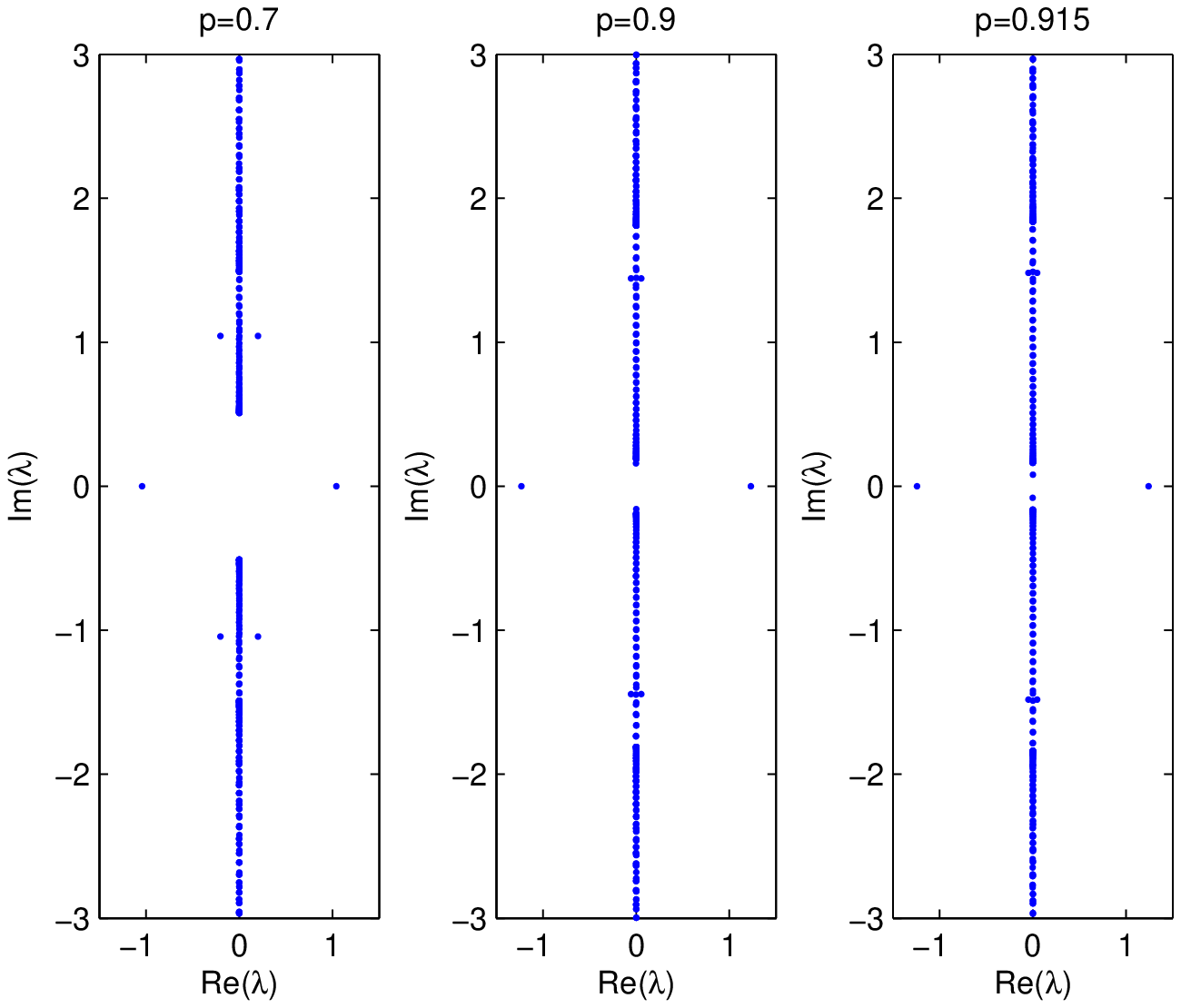}
        \end{subfigure}

         \begin{subfigure}[htbp]{0.35\textwidth}
              \includegraphics[scale=0.55]{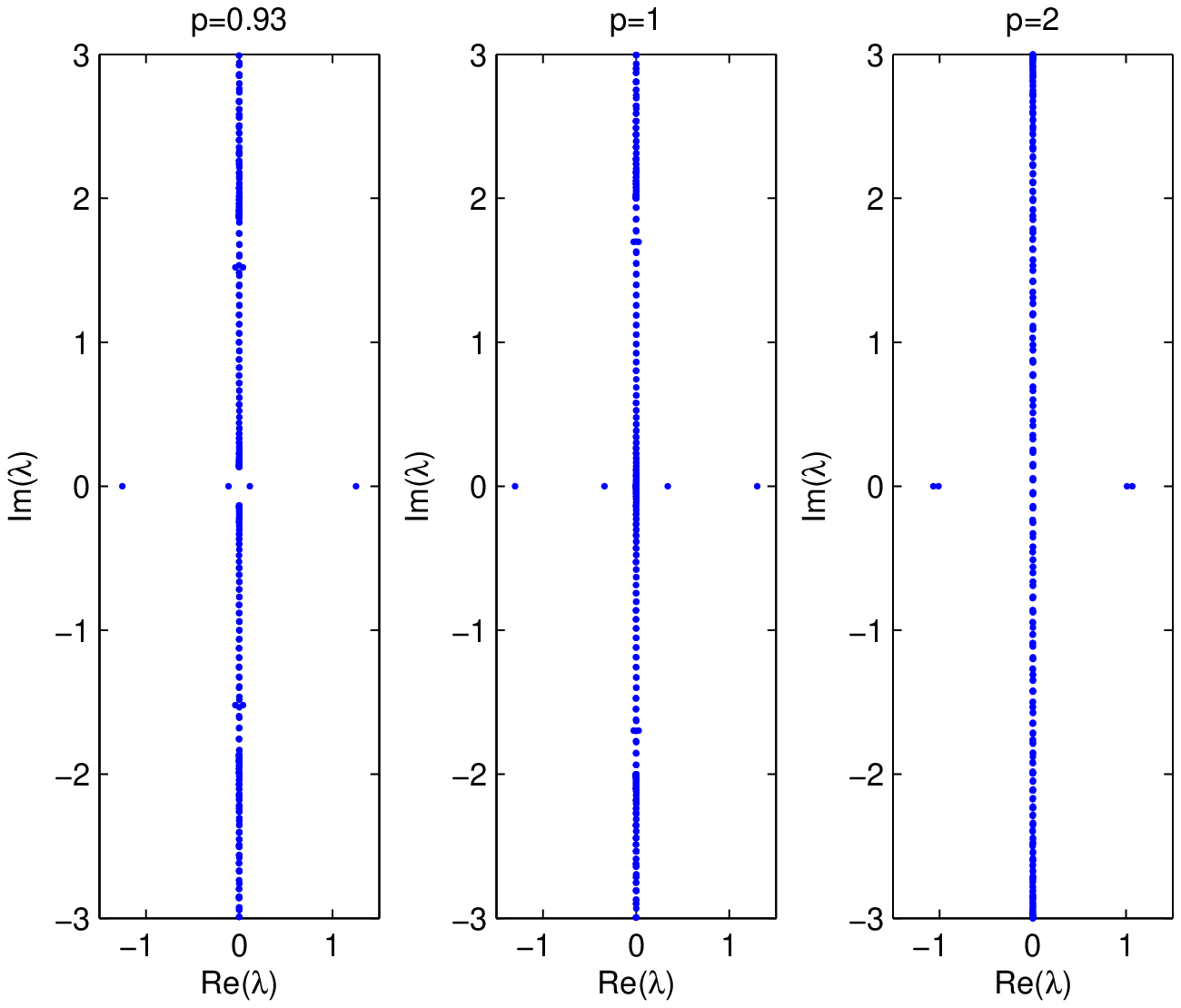}
        \end{subfigure}
             \caption{Numerical approximations for the spectral problem \eqref{eqMTM} associated with the solitary wave
              of the massive Thirring model at $\omega = 0$.}
             \label{MTMFig}
\end{figure}

\begin{figure}[htbp]

       \begin{subfigure}[htbp]{0.5 \linewidth}
              \includegraphics[scale=0.5]{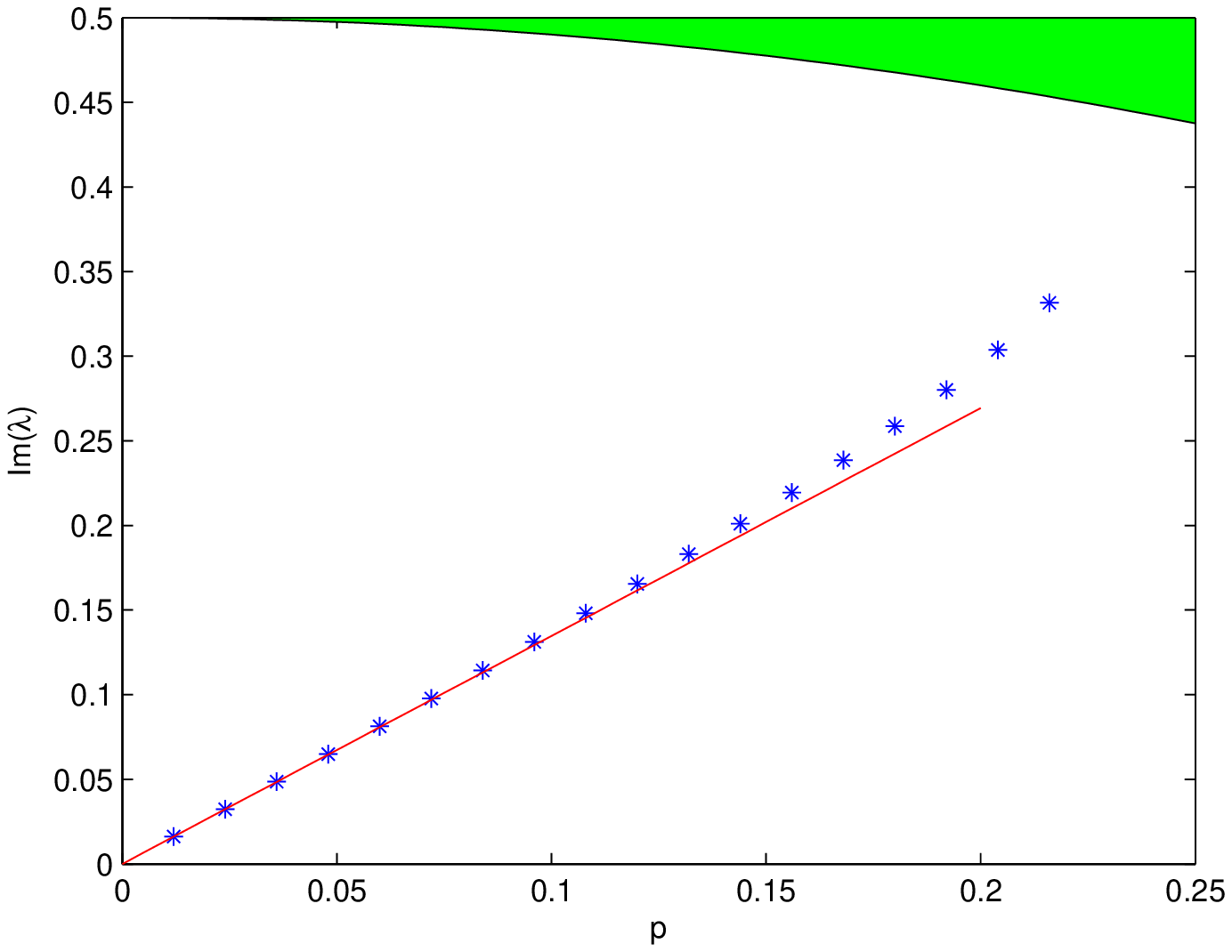}
                \caption{$\omega=0.5$}
                \label{MTM_Imag_05}
        \end{subfigure}
	\begin{subfigure}[htbp]{0.5 \linewidth}
              \includegraphics[scale=0.5]{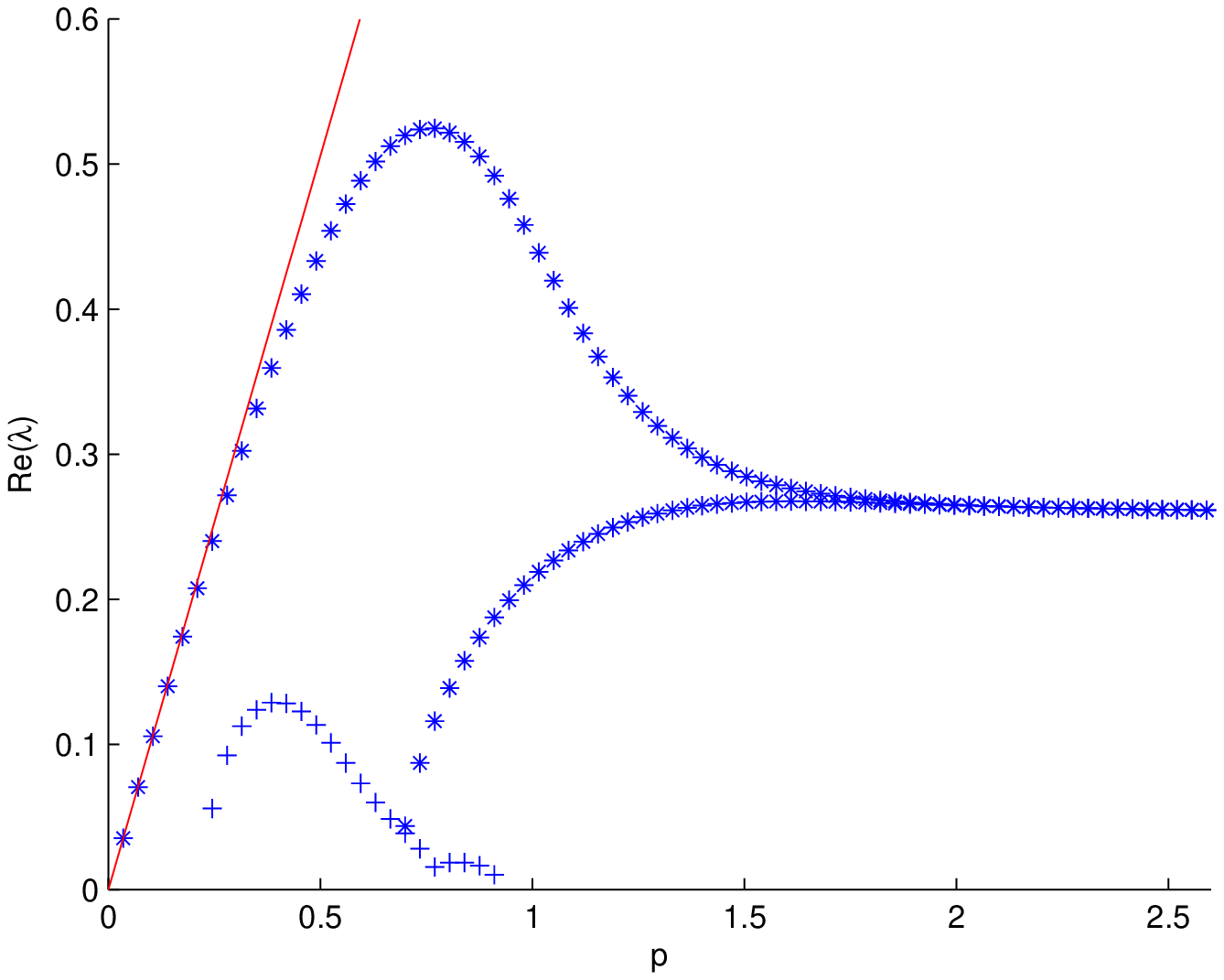}
                \caption{$\omega=0.5$}
                \label{MTM_Real_05}
        \end{subfigure}

        \begin{subfigure}[htbp]{0.5\linewidth}
               \includegraphics[scale=0.5]{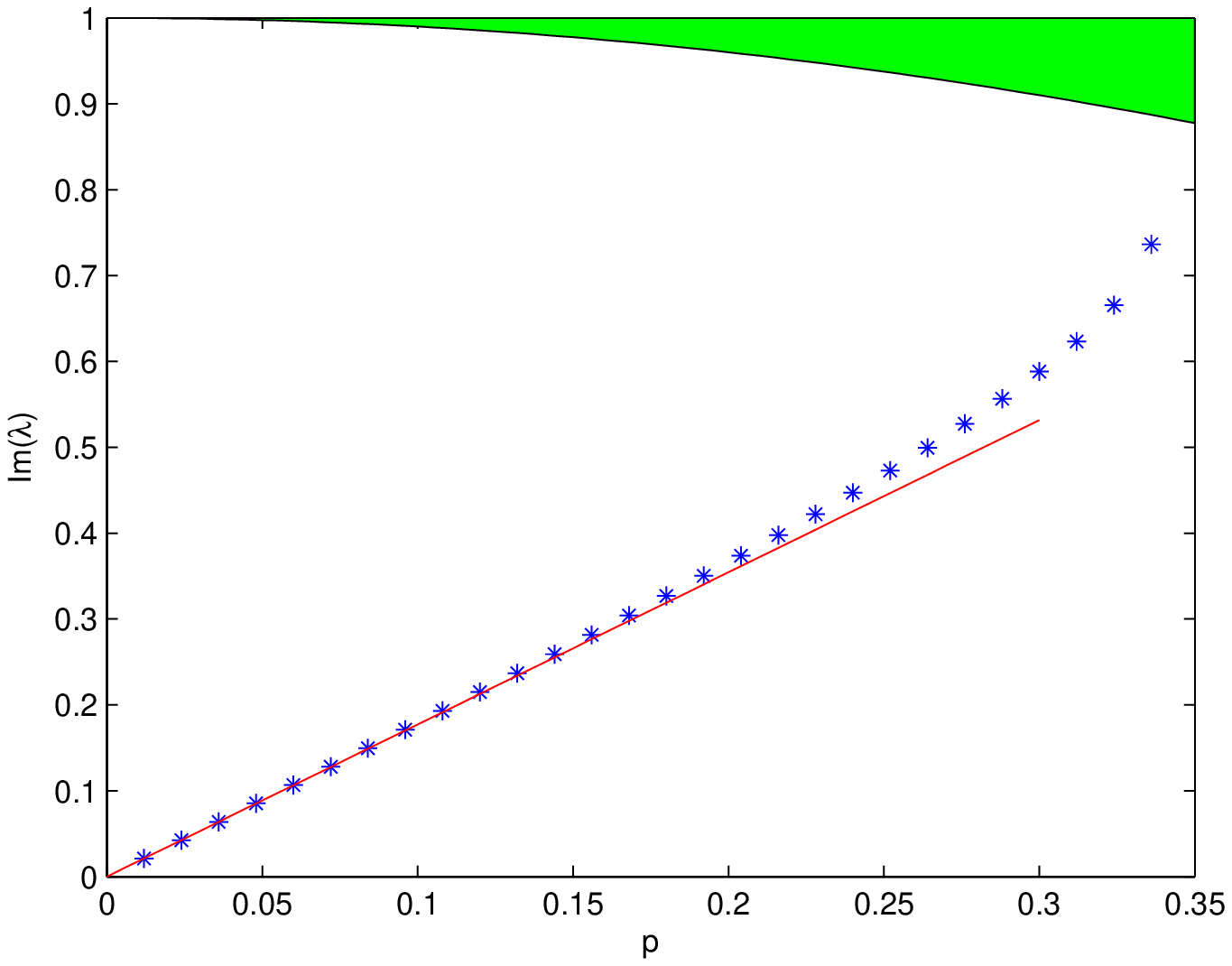}
                \caption{$\omega = 0$}
                \label{MTM_Imag_0}
        \end{subfigure}
        \begin{subfigure}[htbp]{0.5\linewidth}
              \includegraphics[scale=0.5]{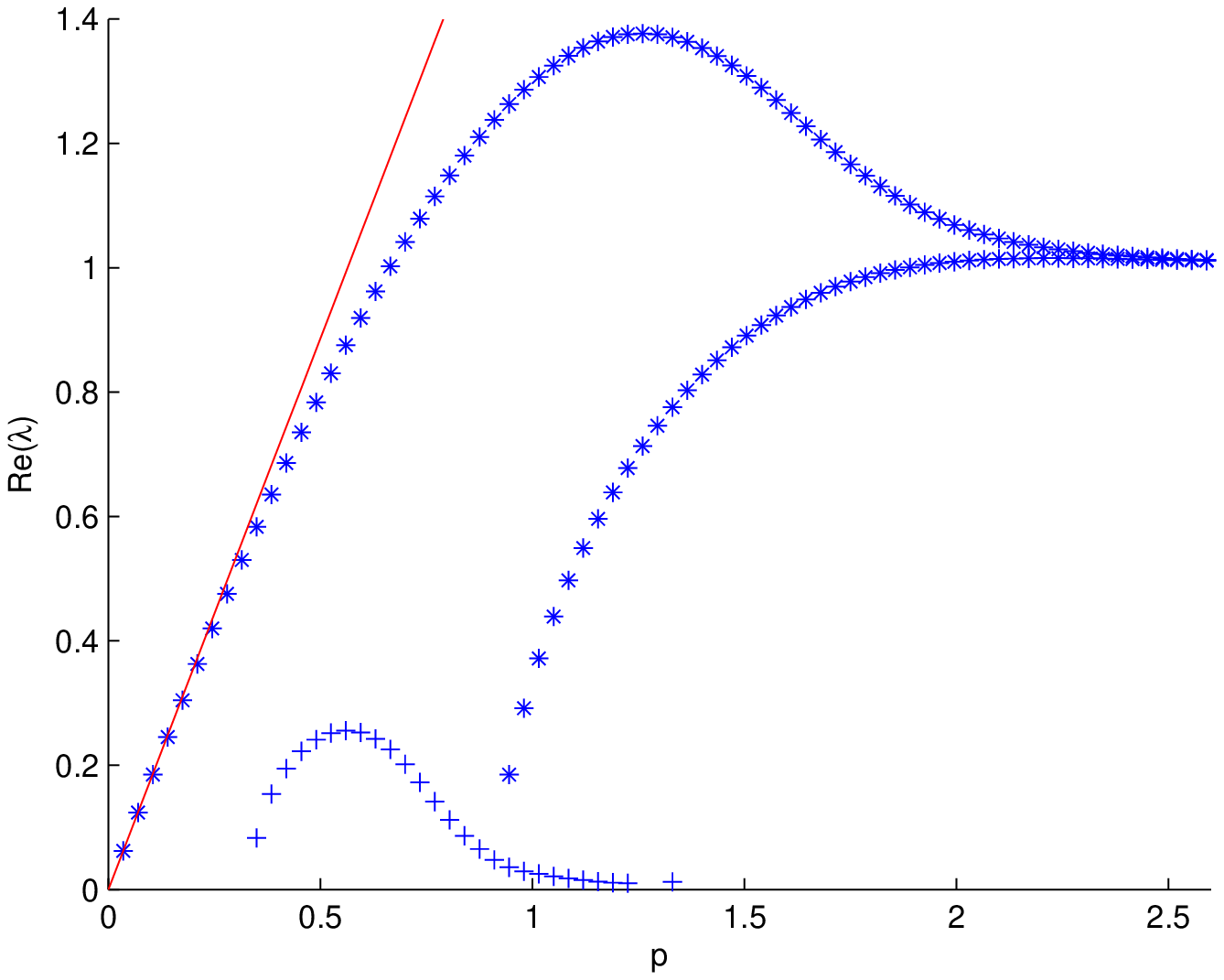}
                \caption{$\omega = 0$}
                \label{MTM_Real_0}
        \end{subfigure}

         \begin{subfigure}[htbp]{0.5\linewidth}
              \includegraphics[scale=0.5]{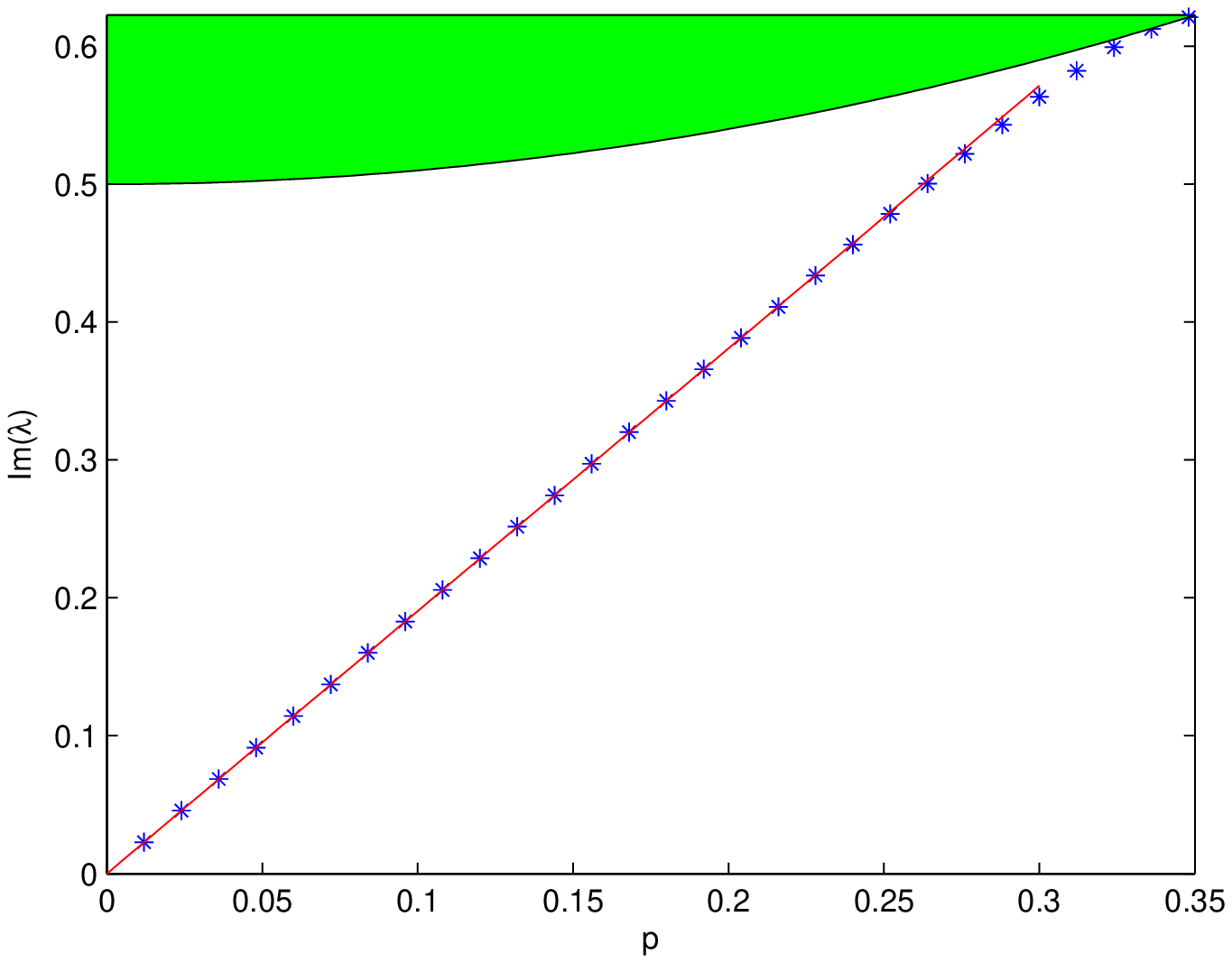}
                \caption{$\omega=-0.5$}
                \label{MTM_Imag_minus05}
        \end{subfigure}
	\begin{subfigure}[htbp]{0.5\linewidth}
              \includegraphics[scale=0.5]{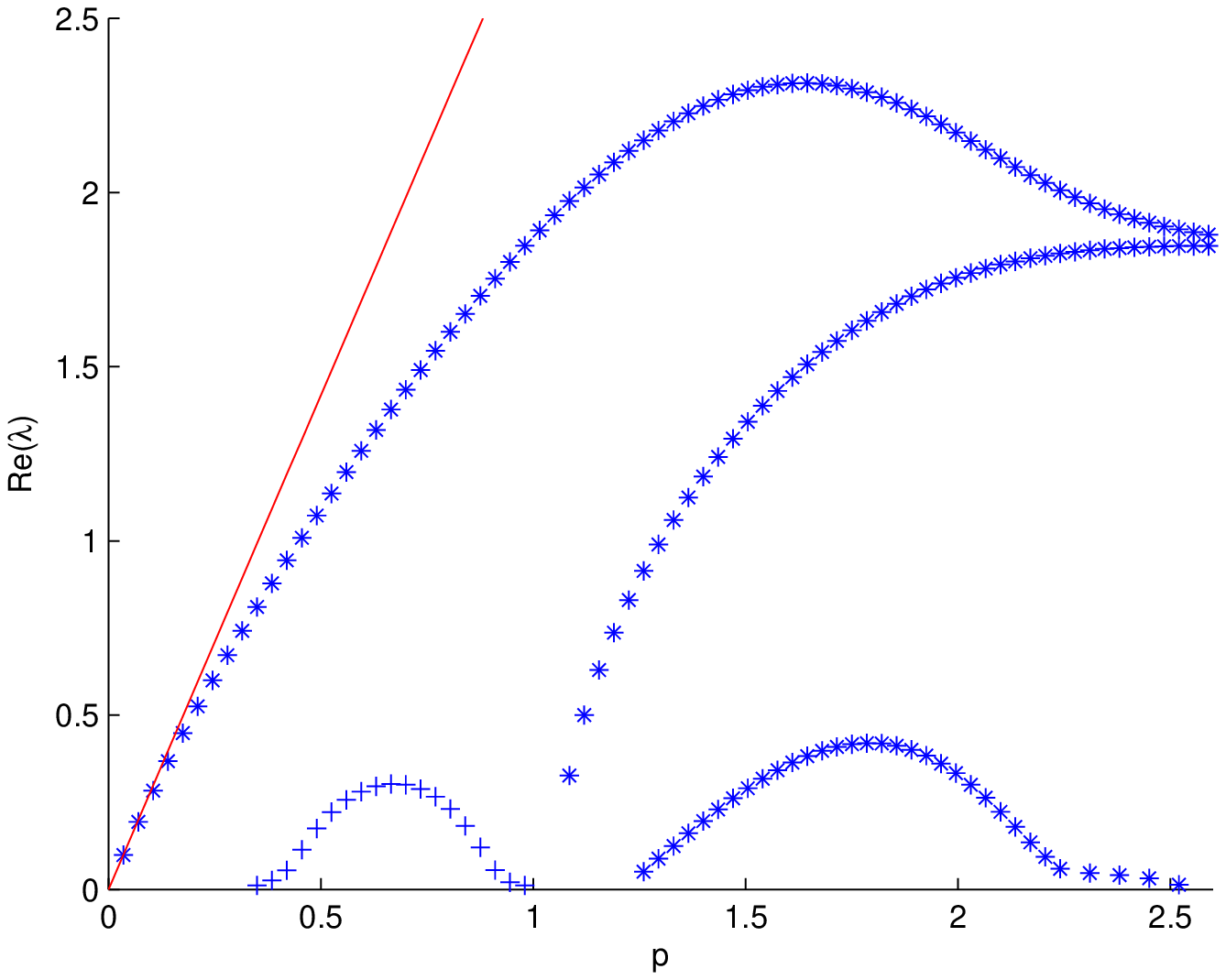}
                \caption{$\omega = -0.5$}
                \label{MTM_Real_minus05}
        \end{subfigure}
             \caption{Numerical approximations of isolated eigenvalues of the spectral problem (\ref{eqMTM}) versus parameter $p$.}
             \label{MTM_Real_Imag}
\end{figure}

\begin{figure}[htbp]
            \begin{subfigure}[htbp]{0.5\linewidth}
               \centering
               \includegraphics[scale=0.5]{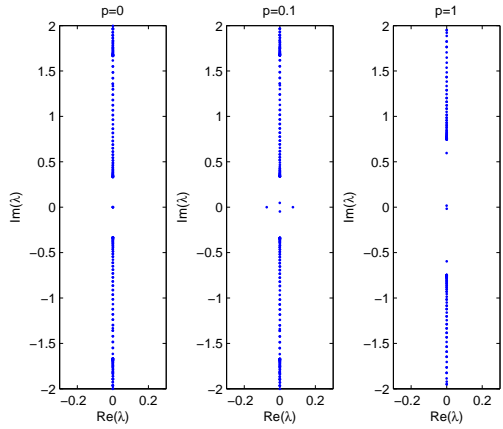}
                \caption{$\omega = 2/3$}
                \label{SolerFig1}
        \end{subfigure}
            \begin{subfigure}[htbp]{0.5\linewidth}
               \includegraphics[scale=0.5]{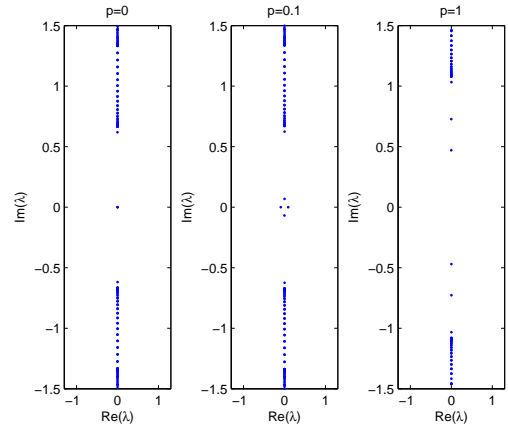}
                \caption{$\omega = 1/3$}
                \label{GrossSubplots}
        \end{subfigure}
        \caption{Numerical approximations for the spectral problem \eqref{eqH}
        associated with the solitary wave of the massive Gross-Neveu model.}
             \label{SolerFig}
\end{figure}

\begin{figure}[htbp]

         \begin{subfigure}[htbp]{0.5\linewidth}
              \includegraphics[scale=0.45]{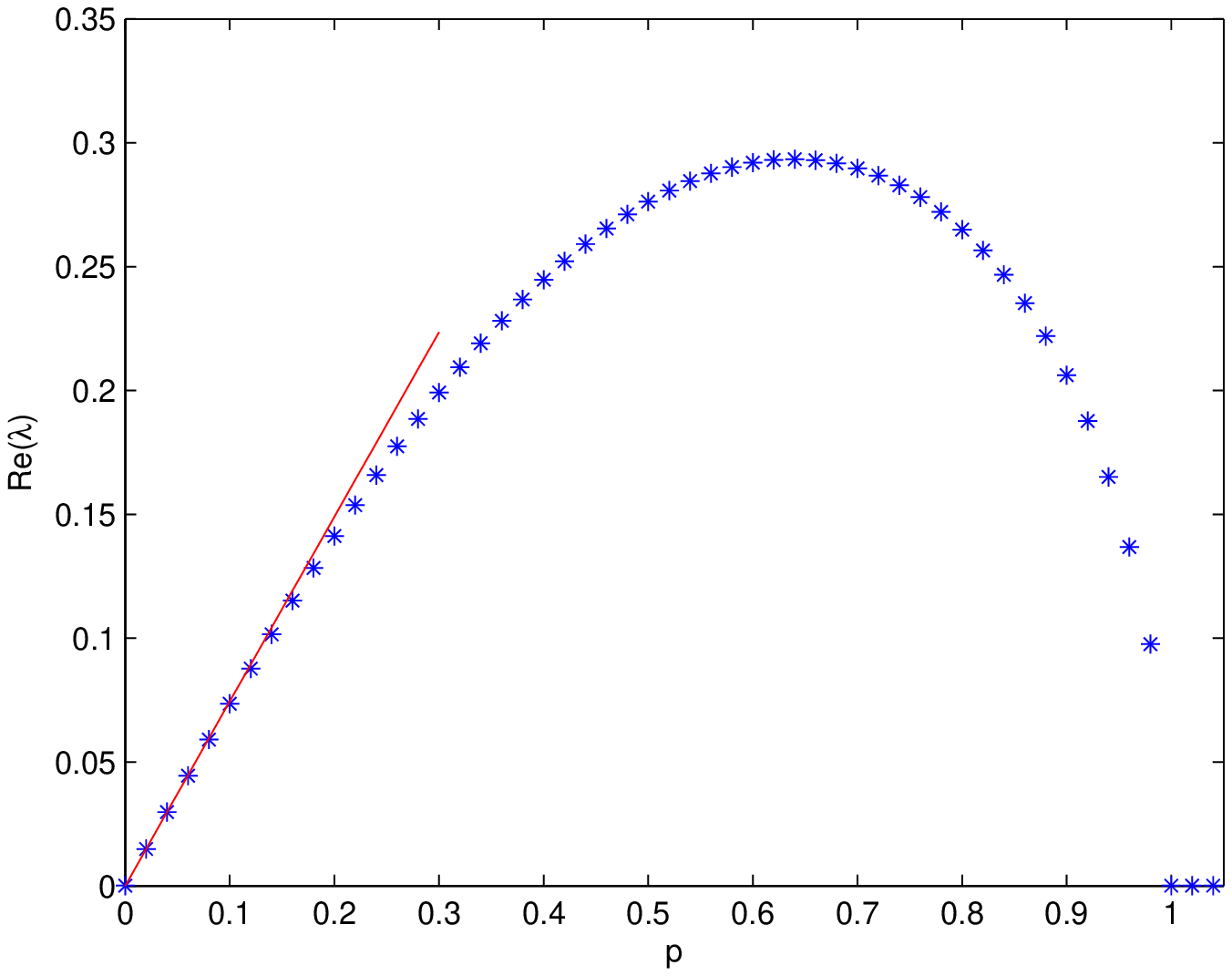}
                \caption{$\omega = 2/3$}
                \label{SolerFig_real}
        \end{subfigure}
            \begin{subfigure}[htbp]{0.5\linewidth}

              \includegraphics[scale=0.45]{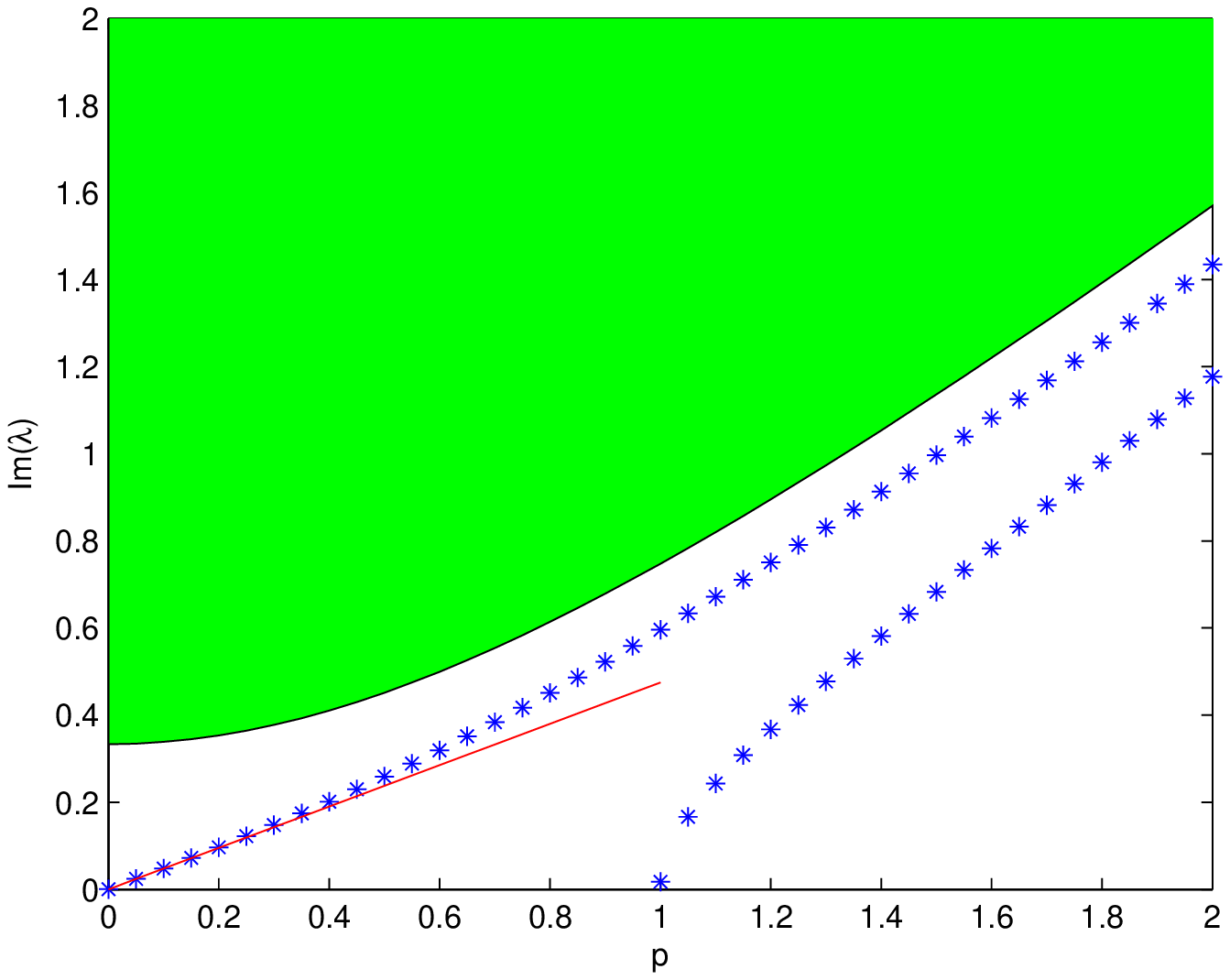}
                \caption{$\omega = 2/3$}
                \label{SolerFig_imag}
        \end{subfigure}

         \begin{subfigure}[htbp]{0.5\linewidth}

              \includegraphics[scale=0.45]{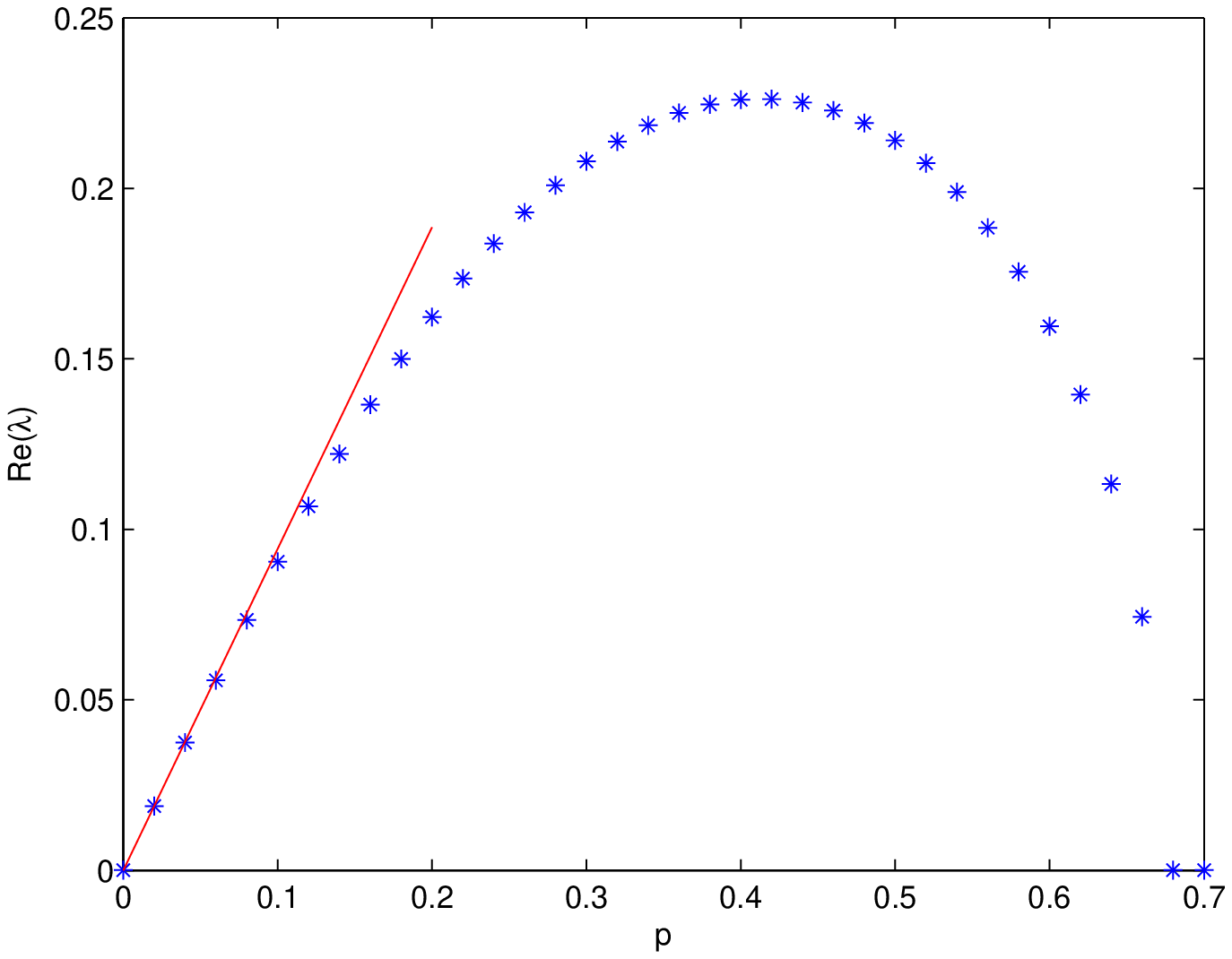}
                \caption{$\omega = 1/3$}
                \label{Gross_real}
        \end{subfigure}
            \begin{subfigure}[htbp]{0.5\linewidth}
              \includegraphics[scale=0.45]{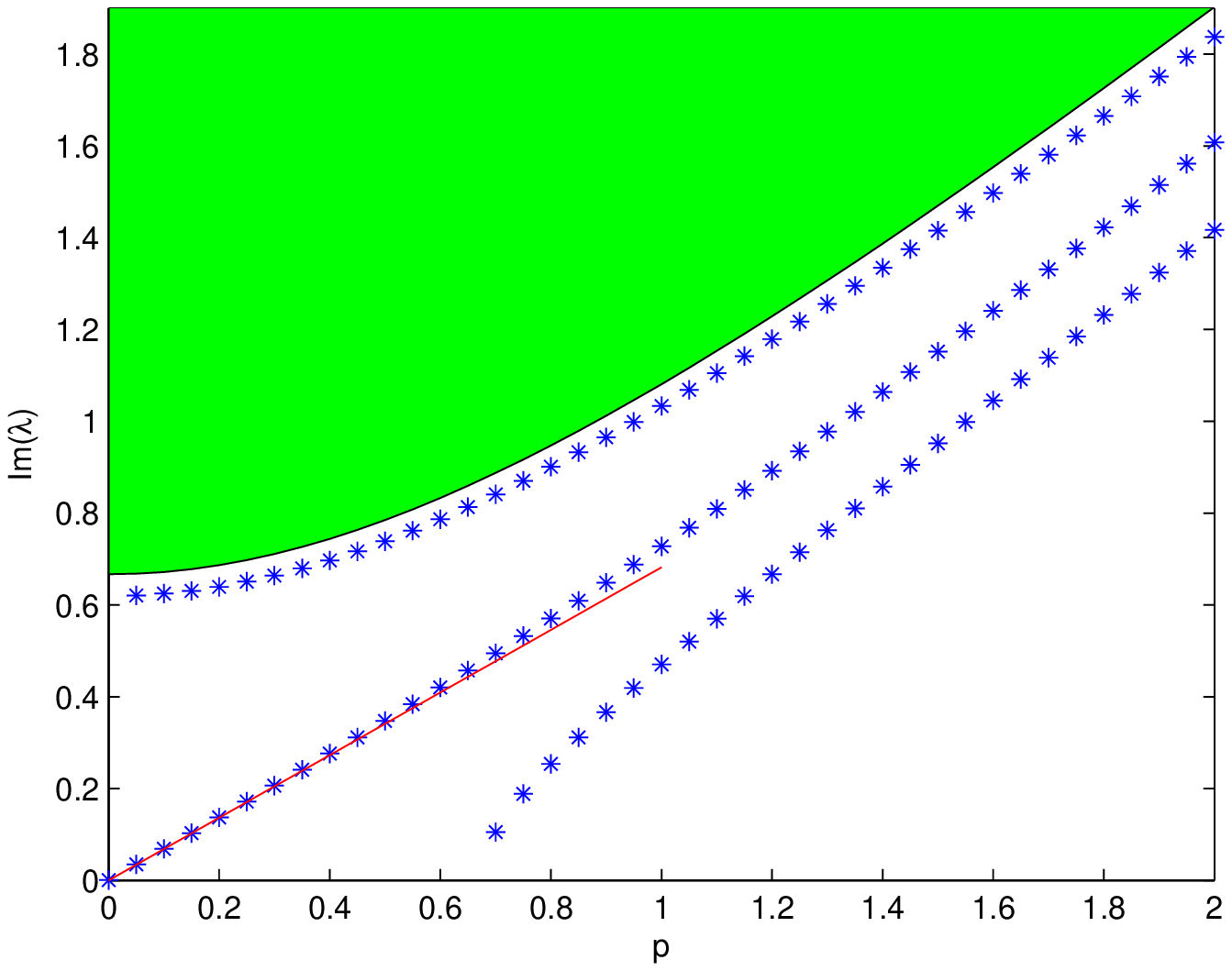}
                \caption{$\omega = 1/3$}
                \label{Gross_imag}
        \end{subfigure}
             \caption{Numerical approximations of isolated eigenvalues of the spectral problem (\ref{eqH}) versus parameter $p$.}
             \label{GrossFig}
\end{figure}

Figure \ref{MTM_Real_Imag} show how the positive imaginary and real eigenvalues
bifurcating from the zero eigenvalue depends on $p$ for $\omega = 0.5, 0, -0.5$,
respectively at each row.
Red solid lines show asymptotic approximations established in
Lemma \ref{lemmaMTM} for $\lambda = \Lambda_r(\omega)p$ and $\lambda = i \Lambda_i(\omega)p$.
Green filled regions in Figures (\ref{MTM_Imag_05}), (\ref{MTM_Imag_0}), and (\ref{MTM_Imag_minus05}) denote the
location of the continuous spectrum. Symbols $*$ and $+$ in Figures (\ref{MTM_Real_05}), 
(\ref{MTM_Real_0}), and (\ref{MTM_Real_minus05}) denote purely real eigenvalues and eigenvalues with nonzero imaginary part.

Numerical results suggest the persistence of transverse instability for any period $p$
because of purely real eigenvalues, which come close to each other and persist
for a large $p$. We observe a stronger instability for a larger solitary wave with
$\omega=-0.5$ than for a smaller solitary wave with $\omega = 0.5$. We notice that
an imaginary eigenvalue does not reach the edge of the continuous spectrum for $\omega =0.5$ and $\omega = 0$
due to colliding with other imaginary eigenvalue coming from the edge of the continuous spectrum.
On the other hand, an imaginary eigenvalue for $\omega=-0.5$ gets absorbed in the edge
of the continuous spectrum. This is explained by the movement of the two branches of the continuous spectrum
in the opposite directions: up and down as the value of $p$ varies. Moving-down branch on $\mbox{Im}(\lambda) > 0$,
as seen in $\omega=0.5$ and $\omega = 0$, expels an eigenvalue from its edge that makes collision with the other
imaginary eigenvalue, while moving-up branch on $\mbox{Im}(\lambda) > 0$, as seen in $\omega =-0.5$,
absorbs an imaginary eigenvalue approaching the edge.

\begin{table}[htdp]
\begin{center}
\begin{tabular}{|c|c|c|c|}
\hline
 & $\omega = -0.5$ & $\omega = 0$ & $\omega=0.5$\\
\hline
$N=100$ & $1.96\times 10^{-1}$ & $2.57\times 10^{-1} $ & $1.16\times 10^{-1}$\\
$N=300$ & $1.36 \times 10^{-4}$ & $2.18\times 10^{-4} $ & $7.02\times 10^{-5} $\\
$N=500$ & $2.22\times 10^{-7}$ & $8.77\times 10^{-5} $ & $6.56 \times 10^{-8}$\\
\hline
\end{tabular}
\end{center}
\caption{$M_{|\Imag\lambda|<10}=\max|\Real\lambda|$ versus values of $\omega$ and $N$
in the case of the spectral problem \eqref{eqMTM} with $p=0$.}
\label{TableMTM}
\end{table}%

To verify a reasonable accuracy of the numerical method, we measure the maximum real part of eigenvalues
along the imaginary axis with $|\mbox{Im}(\lambda)|<10$ and denote it by $M_{|\Imag\lambda|<10}$.
This quantity shows the level of spurious parts of the eigenvalues and it is known to be
large in the finite-difference methods applied to the linearized Dirac systems (see discussion
in \cite{CP}). The following table \ref{TableMTM} shows values of $M_{|\Imag\lambda|<10}$
for three values of $\omega$ and three values of the number $N$ of the Chebyshev points.
In all numerical computations reported on Figures \ref{MTMFig} and \ref{MTM_Real_Imag},
we choose $N=300$, in this way spurious eigenvalues are hardly visible on the figures.

Figures \ref{SolerFig} and \ref{GrossFig} show eigenvalues of the spectral stability problem
(\ref{eqH}) for the solitary wave of the massive Gross--Neveu equation with parameter values
$\omega = 2/3$ and $\omega = 1/3$, respectively. We confirm spectral stability
of the solitary wave for $p = 0$. In agreement with numerical results in \cite{Comech},
we also observe that the spectrum of a linearized operator for $p = 0$ has an additional
pair of imaginary eigenvalues in the case $\omega = 1/3$. (Recall that this issue is contradictory
in the literature with some results reporting spectral instability of solitary waves for $\omega = 1/3$ \cite{Saxena1,Saxena2}.)

The subfigures of Figure \ref{SolerFig} at $p=0.1$ demonstrate our analytical result in Lemma \ref{lemmaSoler},
which predicts splitting of the zero eigenvalue of algebraic multiplicity four
into two pairs of eigenvalues along the real and imaginary axes. Note that the
pair along the real axis persists as the pair of real eigenvalues up to the numerical
accuracy. (Recall that the statement of Lemma \ref{lemmaSoler} lacks the result on
the persistence of real eigenvalues.) Increasing the values of $p$ further, we observe
that the real eigenvalues move back to the origin and split along the imaginary axis, as seen
on the subfigures at $p = 1$.  The gap of the continuous spectrum branches around
the origin is preserved for all values of parameter $p$. The pairs of imaginary eigenvalues
persist in the gap of continuous spectrum for larger values of the parameter $p$.

Figure \ref{GrossFig} shows real and imaginary eigenvalues versus $p$ for the same cases $\omega = 2/3$
and $\omega = 1/3$. The green shaded region indicates the location of the continuous spectrum.
Red solid lines show asymptotic approximations established in
Lemma \ref{lemmaSoler} for $\lambda = \Lambda_r(\omega)p$ and $\lambda = i \Lambda_i(\omega)p$.
It follows from our numerical results that the transverse instability has a threshold on the
$p$ values so that the solitary waves are spectrally stable for sufficiently large values of $p$.
These thresholds on the transverse instability were observed for other values of $\omega$ in $(0,1)$.

To control the accuracy of the numerical method, we again compute the values of $M_{|\Imag\lambda|<10}$
for spurious parts of eigenvalues along the imaginary axis.
Table \ref{TableGross} shows the result for two values of $\omega$ and three values of $N$.
Compared to the case of the massive Thirring model in Table \ref{TableMTM},
we observe a slower convergence rate and lower accuracy of our numerical approximations.

We found that spurious eigenvalues are more visible for smaller values of $\omega$,
in particular, for the value $\omega =1/3$, evidenced in Figure \ref{GrossFigError}.
While spurious eigenvalues in the case of $\omega=1/3$ in Figure \ref{GrossFigError} are quite visible,
the maximum real part of eigenvalues with $|\Imag \lambda|<2$ is at the order of $10^{-3}$
for $N=400$. As a result, the value $N=400$ was chosen for numerical
approximations reported on Figures \ref{SolerFig} and \ref{GrossFig},
this choice guarantees that spurious eigenvalues are hardly visible on the figures.

\begin{table}[htdp]
\begin{center}
\begin{tabular}{|c|c|c|}
\hline
 & $\omega = 1/3$ & $\omega = 2/3$ \\
\hline
$N=100$ & $6.48\times 10^{-2}$ & $2.03\times 10^{-3} $\\
$N=300$ & $1.72\times 10^{-2}$ & $1.68\times 10^{-3} $ \\
$N=500$ & $1.38\times 10^{-2}$ & $1.20\times 10^{-3} $ \\
\hline
\end{tabular}
\end{center}
\caption{$M_{|\Imag\lambda|<10}=\max|\Real\lambda|$ versus values of $\omega$ and $N$
in the case of the spectral problem \eqref{eqH} with $p=0$.}
\label{TableGross}
\end{table}%

\begin{figure}[htbp]
            \begin{subfigure}[htbp]{0.5 \linewidth}
               \includegraphics[scale=0.52]{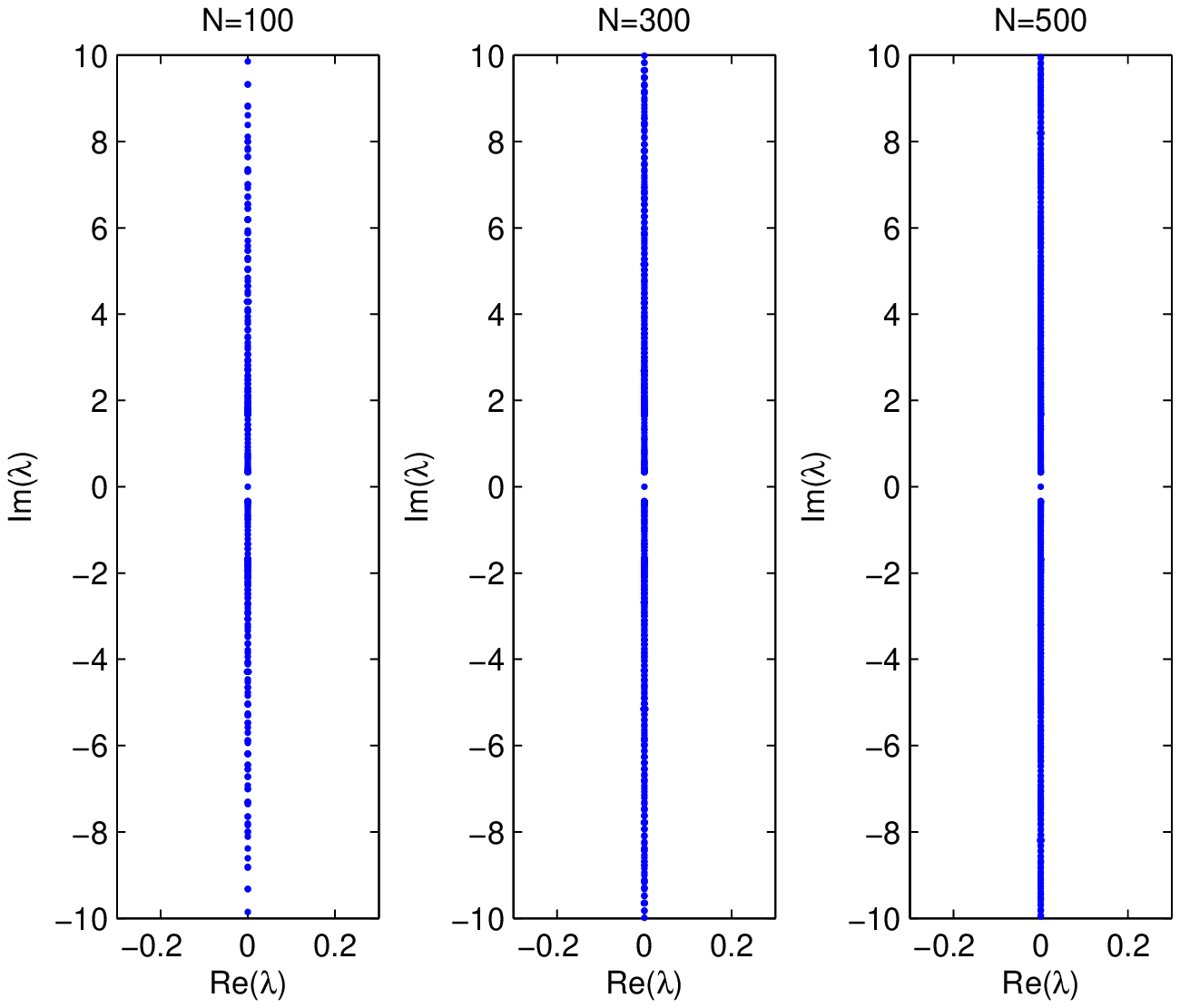}
                \caption{$\omega = 2/3$}
                                \label{GrossError2}
        \end{subfigure}
                \begin{subfigure}[htbp]{0.5 \linewidth}
               \includegraphics[scale=0.52]{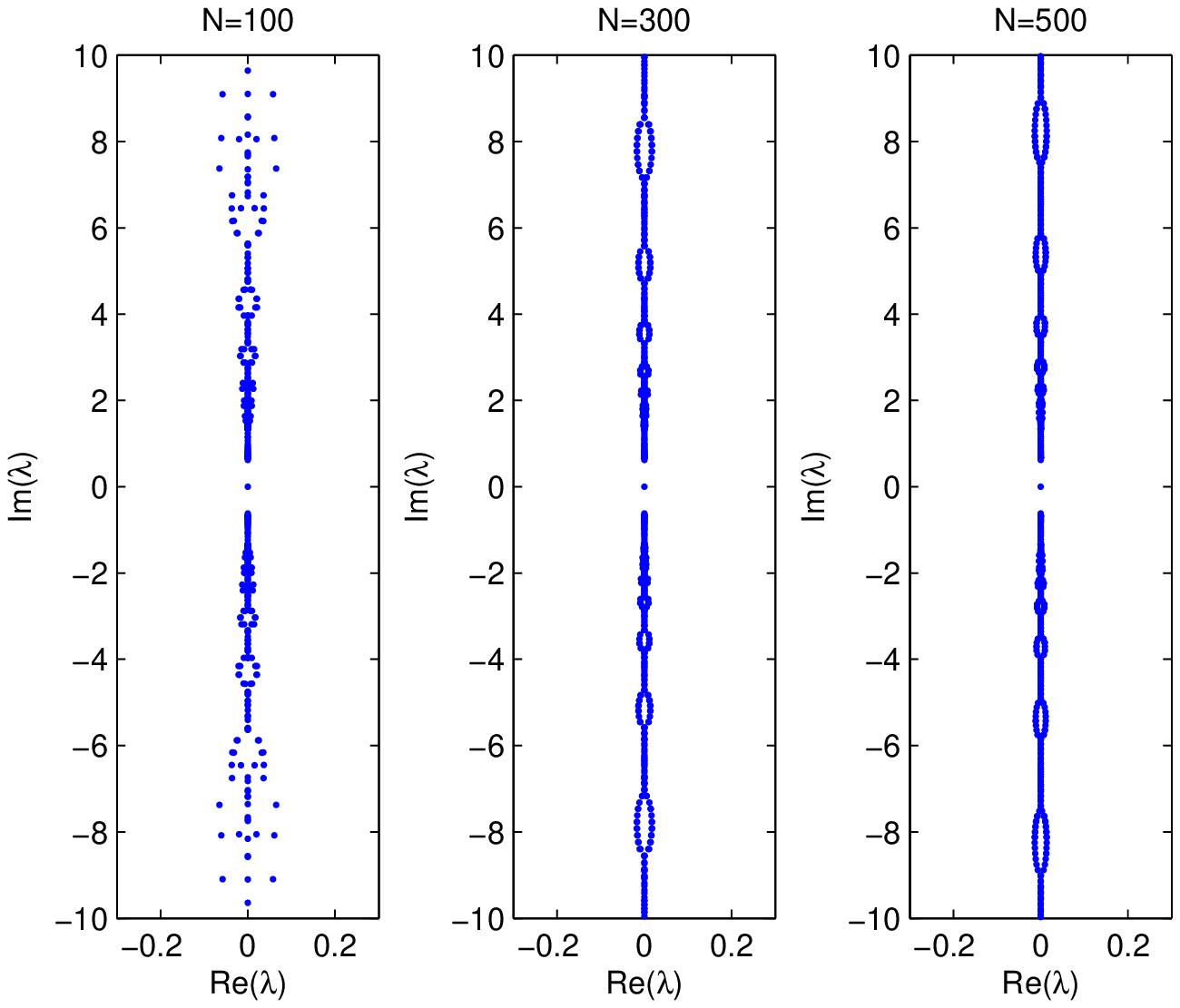}
                \caption{$\omega = 1/3$}
                \label{GrossError1}
        \end{subfigure}
        \caption{Numerically computed $\lambda$ for the spectral problem \eqref{eqH} with $p=0$ 
        for different values of the number $N$ of Chebyshev points.}
             \label{GrossFigError}
\end{figure}

\section{Discussion}

In this last section, we discuss our main result, Theorem \ref{specThm}, in
connection with the more general massive Dirac equations given by the systems
(\ref{dirac1}) and (\ref{dirac3}). One way to consider the more general case
without going too much into technical details is to study reductions
of the massive Dirac equations to the nonlinear Schr\"{o}dinger (NLS) equation.
Both families of solitary waves (\ref{MTM-soliton}) and (\ref{Soler-soliton})
have reductions to the NLS solitary wave in the limit of $\omega \to 1$.
Here we explore a more general reduction to the NLS equation, which is also valid for perturbations
of the solitary wave. Justification of these reductions to the NLS equation
(in a more complex setting of infinitely many coupled NLS equations) can be
found in the recent work \cite{PSW}.

Starting with the massive Dirac equations (\ref{dirac1}) for the periodic stripe potentials,
we can use the scaling transformation
\begin{equation}
\label{ansatz-NLS}
\left\{ \begin{array}{l}
u(x,y,t) = \epsilon e^{it} U(\epsilon x, \epsilon y, \epsilon^2 t), \\
v(x,y,t) = \epsilon e^{it} V(\epsilon x, \epsilon y, \epsilon^2 t),
\end{array} \right.
\end{equation}
where $\epsilon$ is a formal small parameter, and rewrite the system
in the equivalent form
\begin{equation}
\label{dirac1-end}
\left\{ \begin{array}{ll} V - U + i \epsilon U_X + \epsilon^2 (i U_T + U_{YY}) = \epsilon^2 (\alpha_1 |U|^2 + \alpha_2 |V|^2) U, \\
U - V - i \epsilon V_X + \epsilon^2 (i V_T + V_{YY}) = \epsilon^2 (\alpha_2 |U|^2 + \alpha_1 |V|^2) V, \end{array} \right.
\end{equation}
where $X = \epsilon x$, $Y = \epsilon y$, and $T = \epsilon^2 t$ are rescaled variables for slowly varying 
spatial and temporal coordinates. Proceeding now with formal expansions,
$$
\left\{ \begin{array}{l}
U = W + \frac{i}{2} \epsilon W_X + \epsilon^2 \tilde{U}, \\
V = W - \frac{i}{2} \epsilon W_X + \epsilon^2 \tilde{V}, \end{array} \right.
$$
where $W$ is the leading-order part and $(\tilde{U},\tilde{V})$ are correction terms, we obtain
the NLS equation on $W$ at the leading order from the condition that the correction terms  $(\tilde{U},\tilde{V})$
are bounded:
\begin{equation}
\label{NLS-1}
i W_T - \frac{1}{2} W_{XX} + W_{YY} = (\alpha_1 + \alpha_2) |W|^2 W.
\end{equation}
The NLS equation (\ref{NLS-1}) is referred to as the hyperbolic NLS equation because of the
linear diffractive terms. It admits the family of $Y$-independent line solitons if $\alpha_1 + \alpha_2 > 0$,
which includes both the case of the periodic (stripe) potentials with $\alpha_2 = 2 \alpha_1 > 0$ and
the case of the massive Thirring model with $\alpha_1 = 0$ and $\alpha_2 > 0$.

From the previous literature,
see, e.g., recent works \cite{PEOB,PY} or pioneer work \cite{ZR}, it is known that the line solitons
are unstable in the hyperbolic NLS equation with respect to the spatial translation, in agreement
with the result of Lemma \ref{lemmaMTM}. Moreover, the instability region extends to all 
values of the transverse wave number $p$, in agreement with our numerical results 
on Figures \ref{MTMFig} and \ref{MTM_Real_Imag}.  Thus, we anticipate that our results 
are applicable to the more general family of the
massive Dirac equations (\ref{dirac1}) with $\alpha_1 + \alpha_2 > 0$.

Turning now to the massive Dirac equations (\ref{dirac3}) for the hexagonal potentials,
we can use the same scaling transformation (\ref{ansatz-NLS}) and obtain
\begin{equation}
\label{dirac3-end}
\left\{ \begin{array}{ll} V-U + \epsilon (i U_X + V_Y) + i \epsilon^2 U_T =
\epsilon^2 \left(\beta_1(U|U|^2+\overline{U}V^2+2U|V|^2) + \beta_2\overline{U}(U^2-V^2)\right),\\
U - V - \epsilon (i V_X + U_Y) + i \epsilon^2 V_T = \epsilon^2 \left(
\beta_1(V|V|^2+\overline{V}U^2+2V|U|^2) + \beta_2\overline{V}(V^2-U^2)\right). \end{array} \right.
\end{equation}
Proceeding now with formal expansions,
$$
\left\{ \begin{array}{l}
U = W + \frac{\epsilon}{2} (i W_X + W_Y) + \epsilon^2 \tilde{U}, \\
V = W - \frac{\epsilon}{2} (i W_X + W_Y) + \epsilon^2 \tilde{V}, \end{array} \right.
$$
we obtain the following NLS equation for $W$ from the condition that the correction terms
$(\tilde{U},\tilde{V})$ are bounded:
\begin{equation}
\label{NLS-3}
i W_T - \frac{1}{2} W_{XX} -\frac{1}{2} W_{YY} = 4 \beta_1 |W|^2 W.
\end{equation}
The NLS equation (\ref{NLS-3}) is referred to as the elliptic NLS equation because of the
linear diffractive terms. It admits the family of $Y$-independent line solitons if $\beta_1 > 0$,
which includes both the case of the hexagonal potentials with $\beta_1,\beta_2 > 0$ and
the case of the massive Gross--Neveu model with $\beta_1 = - \beta_2 > 0$.

It is well-known from the previous literature,
see, e.g., \cite{KivPel,PY,ZR}, that the line solitons
are unstable in the elliptic NLS equation with respect to the gauge rotation, in agreement
with the result of Lemma \ref{lemmaSoler}. Moreover, the instability band has a finite threshold
on the transverse wave number $p$, in agreement with our numerical results 
on Figures \ref{SolerFig} and \ref{GrossFig}.
Thus, we anticipate that our results are applicable to the more general family of the
massive Dirac equations (\ref{dirac3}) with $\beta_1 > 0$ and arbitrary $\beta_2$.

To summarize, we proved analytically for the massive Thirring and Gross--Neveu models that
the line solitons are unstable with respect to the transverse perturbations
of sufficiently long periods. We approximated eigenvalues of the transverse stability problem
numerically  and showed that the instability region extends to the transverse 
perturbations of any period for
the massive Thirring model but it has a finite threshold for the massive Gross--Neveu model.
Justified with the small-amplitude reduction to the hyperbolic or elliptic NLS equations,
we extended this conclusion to the more general massive Dirac equations which model
periodic stripe and hexagonal potentials in the two-dimensional Gross--Pitaevskii equation.

\vspace{0.5cm}

{\bf Acknowledgements:} The authors thank A. Comech and P. Kevrekidis for useful discussions
of the numerical part of this work. D.P. is supported by the NSERC grant. Y.S. is supported by
a graduate scholarship of the Department of Mathematics at McMaster University.


\begin{thebibliography}{10}

\bibitem{Abl1} M.J. Ablowitz and Y. Zhu, ``Nonlinear waves in shallow honeycomb lattices",
SIAM J. Appl. Math. {\bf 72}  (2012),  240--260.

\bibitem{Abl2} M.J. Ablowitz and Y. Zhu, ``Nonlinear wave packets in deformed honeycomb lattices",
SIAM J. Appl. Math. {\bf 73}  (2013),  1959-–1979.

\bibitem{Comech} G. Berkolaiko and A. Comech, ``On spectral stability of solitary waves
of nonlinear Dirac equations in 1D", Math. Model. Nat. Phenom. {\bf 7} (2012), no. 2, 13--31.

\bibitem{BC} N. Boussaid and A. Comech, ``On spectral stability of the nonlinear Dirac equation",
arXiv:1211.3336 (2012).

\bibitem{Candy} T. Candy, ``Global existence for an $L^2$-critical nonlinear Dirac equation
in one dimension", Adv. Diff. Eqs. {\bf 7-8} (2011), 643--666.

\bibitem{CP} M. Chugunova and D. Pelinovsky, ``Block-diagonalization of the symmetric first-order
coupled mode system", SIAM J. Appl. Dyn. Syst. {\bf 5} (2006), 55--83.


\bibitem{PS2} A. Contreras, D.E. Pelinovsky, and Y. Shimabukuro, ``$L^{2}$ orbital stability of
Dirac solitons in the massive Thirring model", arXiv:1312.1019 (2013)

\bibitem{AD} T. Dohnal and A.B. Aceves, ``Optical soliton bullets in (2+1)D nonlinear Bragg resonant periodic geometries",
Stud. Appl. Math.  {\bf 115}  (2005),  209--232.

\bibitem{FW1} C. L. Fefferman, M. I. Weinstein, ``Honeycomb lattice potentials and Dirac points'',
J. Am. Math. Soc. {\bf  25}, 1169-1220 (2012)

\bibitem{FW2} C. L. Fefferman, M. I. Weinstein, ``Wave packets in honeycomb structures and two-dimensional Dirac equations",
Commun. Math. Phys. {\bf 326}, 251-286 (2014)

\bibitem{FW3} C. L. Fefferman, M. I. Weinstein, ``Waves in honeycomb structures",
Journees Equations aux derivees partialles, Biarretz, 3-7 juin 2012, GDR 243 4 (CNRS).

\bibitem{gross-neveu} D.J. Gross and A. Neveu, ``Dynamical symmetry breaking in
asymptotically free field theories", Phys. Rev. D {\bf 10} (1974), 3235–-3253.

\bibitem{Huh} H. Huh, ``Global solutions to Gross--Neveu equation", Lett. Math. Phys. {\bf 103} (2013), 927--931.

\bibitem{Kev1} Q.E. Hoq, J. Gagnon, P.G. Kevrekidis, B.A. Malomed, D.J. Franzeskakis, and
R. Carretero--Gonz\'ales, ``Extended nonlinear waves in
multidimensional dynamical lattices", Math. Comp. Simul. 80,
721--731 (2009).

\bibitem{Jon1} M.A. Johnson, ``The transverse instability of periodic waves in
Zakharov-Kuznetsov type equations", Stud. Appl. Math. {\bf 124}  (2010),  323--345.

\bibitem{Jon2} M.A. Johnson and K. Zumbrun, ``Transverse instability of periodic traveling waves
in the generalized Kadomtsev-Petviashvili equation", SIAM J. Math. Anal. {\bf 42}  (2010),  2681--2702.


\bibitem{Kev2} P.G. Kevrekidis, J. Gagnon, D.J. Franzeskakis, and B.A. Malomed,
``$X$, $Y$, and $Z$ waves: extended structures in nonlinear
lattices", Phys. Rev. E 75, 016607 (2007).

\bibitem{KivPel} Yu.S. Kivshar and D.E. Pelinovsky, ``Self-focusing and transverse
instabilities of solitary waves", Phys. Rep. {\bf 331} (2000), 117--195.

\bibitem{Saxena1} F.G. Mertens, N.R. Quintero, F. Cooper, A. Khare, and A. Saxena,
``Nonlinear Dirac equation solitary waves in external fields", Phys. Rev. E {\bf 86}, 046602 (2012).

\bibitem{P} D. E. Pelinovsky, {\em Localization in periodic potentials : from Schr\"{o}dinger operators
to the Gross-Pitaevskii equation},
London Mathematical Society lecture note series {\bf 390} (Cambridge University Press, Cambridge, 2011)

\bibitem{Pel-survey} D. Pelinovsky, ``Survey on global existence in the nonlinear Dirac equations in one dimension",
in {\em Harmonic Analysis and Nonlinear Partial Differential Equations}
(Editors T. Ozawa and M. Sugimoto) RIMS Kokyuroku Bessatsu, B {\bf 26} (2011), 37--50.

\bibitem{PRT} D. Pelinovsky, F. Rousset, and N. Tzvetkov, ``Normal form for transverse instability
of line solitons of the Zakharov--Kuznetsov equation", in preparation.

\bibitem{PEOB} D. E. Pelinovsky, E. A. Ruvinskaya, O. A. Kurkina, B. Deconinck, ``Short-wave transverse
instabilities of line solitons of the $2$-D hyperbolic nonlinear Schr\"{o}dinger equation",
Theor. Math. Phys. {\bf 179} (2014), 452--461.

\bibitem{PS} D. Pelinovsky and G. Schneider, ``Bounds on the tight-binding approximation
for the Gross-Pitaevskii equation with a periodic potential", J. Diff. Eqs. {\bf 248} (2010), 837--849.

\bibitem{PS1} D.E. Pelinovsky and Y. Shimabukuro, ``Orbital stability of Dirac solitons",
Lett. Math. Phys. {\bf 104} (2014), 21--41.

\bibitem{PSW} D.E. Pelinovsky, G. Simpson, and M.I. Weinstein, ``Polychromatic solitons
in a periodic and nonlinear Maxwell system",  SIAM J. Appl. Dynam. Syst. {\bf 11} (2012), 478--506.

\bibitem{PY} D.E. Pelinovsky and J. Yang,  ``On transverse stability of discrete line solitons",
Physica D {\bf 255} (2014), 1--11.

\bibitem{RT12} F. Rousset and N. Tzvetkov, ``Stability and instability of
the KDV solitary wave under the KP-I flow", Commun. Math. Phys. {\bf 313} (2012), 155--173.

\bibitem{US} G. Schneider and H. Uecker, ``Nonlinear coupled mode dynamics
in hyperbolic and parabolic periodically structured spatially
extended systems'', Asymp. Anal. {\bf 28}, 163--180 (2001).

\bibitem{Selberg} S. Selberg and A. Tesfahun, ``Low regularity well-posedness for some nonlinear Dirac equations
in one space dimension'', Diff. Integral Eqs. {\bf 23} (2010), 265--278.

\bibitem{Saxena2} S. Shao, N.R. Quintero, F.G. Mertens, F. Cooper, A. Khare, and A. Saxena,
``Stability of solitary waves in the nonlinear Dirac equations with arbitrary nonlinearity",
Phys. Rev. E {\bf 90} (2014), 032915.

\bibitem{Thirring} W. Thirring, ``A soluble relativistic field theory",
Annals of Physics {\bf 3} (1958), 91-–112.

\bibitem{Trefethen} L. N. Trefethen, {\em Spectral methods in MATLAB} (SIAM, Philadelphia, 2000).

\bibitem{KK} Y. Yamazaki, ``Stability of line standing wave near the bifurcation point for nonlinear
Schr\"odinger equations,'' Kodai Math. J. (2014).

\bibitem{Yang1} J. Yang, ``Transversely stable soliton trains in photonic lattice",
Phys. Rev. A 84, 033840 (2011).

\bibitem{Yang2} J. Yang, D. Gallardo, A. Miller, and Z. Chen,
``Elimination of transverse instability in stripe solitons by
one-dimensional lattices", Opt. Lett. 37, 1571--1573 (2012).

\bibitem{ZR} V.E. Zakharov and A.M. Rubenchik, ``Instability of waveguides and solitons in nonlinear media",
Sov. Phys. JETP {\bf 38} (1974), 494-500 (1974).


\bibitem{Zhang} Y. Zhang and Q. Zhao, ``Global solution to nonlinear Dirac equation for
Gross--Neveu model in $1+1$ dimensions, arXiv:1407.4221 (2014).
\end{thebibliography}
\end{document}